\documentclass[11pt]{amsart}
\usepackage{graphicx} 
\usepackage{amsthm}
\usepackage{amssymb}
\usepackage{amsmath}
\usepackage{xcolor}
\usepackage[style=ext-numeric,
            sorting=nyt,
            sortcites=true,
            maxnames=50,
            giveninits=true,
            articlein=false]{biblatex}
\usepackage{physics}
\bibliography{Bibliography}
\usepackage{appendix}

\usepackage{nicematrix}

\usepackage[margin=1.24in
]{geometry} 

\usepackage[mathscr]{euscript}
\usepackage{mathrsfs}

\newtheorem{theorem}{Theorem}[section]
\newtheorem{prop}[theorem]{Proposition}
\newtheorem{corollary}[theorem]{Corollary}
\newtheorem{lemma}[theorem]{Lemma} 

\theoremstyle{definition}
\newtheorem{definition}[theorem]{Definition} 
\newtheorem{exmp}{Example}[section]
\newtheorem{conjecture}{Question}

\theoremstyle{remark}
\newtheorem{remark}[theorem]{Remark} 

\usepackage{verbatim}

\newcommand{\C}{\mathbb{C}}
\newcommand{\R}{\mathbb{R}}

\newcommand{\N}{\mathbb{N}}
\newcommand{\T}{\mathbb{T}}

\title[Extensions of Daubechies' theorem]{Extensions of Daubechies' theorem: Reinhardt domains, Hagedorn wavepackets and mixed-state localization operators}
\author{Erling Svela}
\address{Department of Mathematical Sciences, Norwegian University of Science and Technology, 7491 Trondheim, Norway}
\email{erling.a.t.svela@ntnu.no}
\date{}

\begin{document}
\begin{abstract}
    Daubechies-type theorems for localization operators are established in the multi-variate setting, where Hagedorn wavepackets are identified as the proper substitute of the Hermite functions. 
    The class of Reinhardt domains is shown to be the natural class of masks that allow for a Daubechies-type result. Daubechies' classical theorem is a consequence of double orthogonality results for the short-time Fourier transform. We extend double orthogonality to the quantum setting and use it to establish Daubechies-type theorems for mixed-state localization operators, a key notion of quantum harmonic analysis. Lastly, we connect the results to Toeplitz operators on quantum Gabor spaces.
   
\end{abstract}
\subjclass[2020]{42C49,47G30,47B35,42C05,32Q02} 
\keywords{Localization operators, quantum harmonic analysis, Hermite functions, Hagedorn wavepackets, Reinhardt domains}

\maketitle

\section{Introduction}
Let \(f,g\in L^2(\R^d)\) and \(F:\R^{2d}\rightarrow \C\). We define the (generalized) localization operator with mask \(F\) and window \(g\) by \begin{align}\label{LocOpDef}
    A_F^gf=\int_{\R^{2d}} F(z)\langle f,\pi(z)g\rangle \pi(z)g\;dz,
\end{align}
where \(\pi(z)g(t)=\pi(x,\omega)g(t)=e^{2\pi i \langle t,\omega \rangle}g(t-x)\) and we interpret the integral weakly. First introduced by Daubechies \cite{Daubechies}, localization operators have become key items in time-frequency analysis, and many aspects of their structure have been studied in the literature \cite{FG,Boundedness1,Szego,RamTop,Berezin}. A fundamental result in the theory of localization operators is Daubechies' theorem \cite[Section 4]{Daubechies}. In the case of a polyradial mask \(F\) and the Gaussian window function, \(\phi_0(t)=2^{d/4}e^{-\pi|t|^2}\), Daubechies' theorem states that the eigenfunctions of the localization operator with mask \(F\) and window \(\phi_0\), \(A_F^{\phi_0}\), are the Hermite functions \(\{\phi_n\}_{n\in\N_0^d}\). Moreover, the theorem also provides an explicit formula for the eigenvalues. In the one-dimensional case, the eigenvalue of the \(n\)-th Hermite function is \begin{align*}
    \lambda_n=\int_0^\infty F\left(\sqrt{\frac{u}{\pi}}\right)\frac{r^n}{n!}e^{-r}\;dr.
\end{align*}
 Knowing the complete spectral data of these localization operators is of great use, and many authors have exploited this knowledge to establish a variety of results. Some examples include quantitative uncertainty principles \cite{Helge1,Helge2,Helge3} and inverse problems \cite{AbreuDorfler}.
 In contrast, for general localization operators only asymptotic formulas for the eigenvalues are known \cite{Szego,RamTop}, and the properties of the eigenfunctions depend heavily on the properties of the localization operator \cite{RamTop}. As such, the utility of Daubechies' theorem makes studying localization operators with Gaussian window considerably easier than for general windows.\newline

Localization operators are closely connected to the short-time Fourier transform (STFT). Defined for \(f,g\in L^2(\R^d)\) by \(V_gf(z)=\langle f,\pi(z)g\rangle\), the short-time Fourier transform represents the function \(f\) jointly in time and frequency, unlike the Fourier transform, which separates time and frequency. A natural question is then to ask how much of the representation \(V_gf\) is concentrated on a certain domain in the time-frequency plane. That is, given a domain \(\Omega\subset \R^{2d}\), how large can \(\int_{\Omega} |V_gf(z)|^2\;dz\) be compared to the norm of \(f\)? This can be thought of as a quantitative uncertainty principle for the STFT. More generally, given a function \(F:\R^{2d}\rightarrow \R_+\) and a normalized \(g\in L^2(\R^d)\), determine \begin{align}\label{ConcProblem}
    \sup_{ \|f\|=1} \int_{\R^{2d}}|V_gf(z)|^2 F(z)\;dz.
\end{align}
Ideally, one would also like to know which functions maximize this quantity. Via the min-max principle, the maxima are the eigenvalues of the localization operator \(A_F^g\), and the maximizers are the eigenfunctions. As such, Daubechies' theorem also yields the solution to the spectrogram localization problem as long as the window is a Gaussian and the weight is polyradial. While much progress has been made on the general localization problem \cite{LargeSieve,AccSpec}, the sharpest recent results have all assumed the window to be Gaussian \cite{FaberKrahn,isoperimetric,LocOpNorm}, hinting at Daubechies' theorem as a valuable tool in the study of the localization problem as well. \newline

In this work we look for extensions of Daubechies' theorem. In the one-dimensional case some progress has already been made. It was shown in \cite{HermiteDaub} that if \(F\) is radial and \(g=\phi_n\) is a Hermite function, then the eigenfunctions of \(A_F^{\phi_n}\) are also the Hermite functions. Approaching the eigenvalue problem using the concept of double orthogonality, we generalize the result in \cite{HermiteDaub} to the multivariate case, and to a larger class of polyradial masks. The generalization relies on the characterization of compact localization operators due to Fernández and Galbis \cite{FG}. In the special case where \(F=\chi_{\Omega}\) is a characteristic function, we show that for Daubechies' theorem to hold, \(\Omega\) must be a Reinhardt domain, a well-known class of domains from several complex variables~\cite{SCVBook}. Thus, for Daubechies' theorem it is Reinhardt domains, not balls, which are the natural multivariate analogues of spherically symmetric sets. Moreover, we also prove a similar result for a class of generalized Hermite functions, namely the Hagedorn wavepackets \cite{Hagedorn1}. The eigenvalue problem for localization operators is known to be covariant under symplectic transforms \cite{FaberKrahn,deGossonBook}, and the Hagedorn wavepackets serve as an explicit description of the symplectic covariance in the case of a polyradial mask and a Hermite analysis window.\newline

Localization operators also play an important role in quantum harmonic analysis (QHA), a field of study inspired by Werner's work~\cite{Werner}. In QHA, localization operators are the simplest case of function-operator convolution, namely the convolution between a function \(F\) and a rank-one operator (or pure state) \(g\otimes g\). More concretely, we have\begin{align*}
    A_F^g=F\star (g\otimes g),
\end{align*} where \(\star\) denotes function-operator convolution. A natural generalization of the eigenvalue problem for localization operators is thus to investigate the eigenvalue problem for general function-operator convolutions, which are called mixed-state localization operators \cite{bible2} and denoted by \(F\star S\) for some operator \(S\). While the localization problem for mixed-state localization operators was first studied in \cite{bible2}, several special cases of this problem have seen thorough study. In addition to the classical localization problem for the STFT mentioned above, the localization problem for the Wigner distribution: \begin{align}\label{WignerProblem}
        \sup_{ \|f\|=1} \int_{\R^{2d}} W(f)(z) F(z)\;dz,
\end{align} (see Section~\ref{ssec:WignerWeyl} for the definition of the Wigner distribution) has an equally long history, see for instance~\cite{flandrin,LiebOstrover,Lerner,RamTop2}. In QHA, this problem corresponds to the mixed-state localization operator \(F\star 2^dP\). We contribute to the general theory by introducing a quantum version of double orthogonality, which characterizes the eigenfunctions of a mixed-state localization operator, just like in the pure-state case. While there is no full characterization of compact mixed-state localization operators, we make use of two sufficient conditions for compactness: Fernández-Galbis' criterion in the trace class case, and the Tauberian theorem for operators in the non-trace class case \cite{FG,bible3}. Furthermore, we establish a Daubechies-type theorem for certain mixed-state localization operators, namely the ones where the function is polyradial, and the operator has polyradial Fourier-Wigner transform. The condition on the operator is the natural extension of the pure state condition of the analysis window being a Gaussian or a Hermite function, and also captures previous work on the Weyl transform~\cite{LiebOstrover,RamTop2}. 
We believe our extension of Daubechies' theorem to be of relevance outside of quantum harmonic analysis, and this is illustrated by applications to problems in time-frequency analysis. We first solve a variant of a problem of Lerner about the spectrum of the Weyl transform of rotationally invariant domains~\cite{Lerner}, which is closely related to the Wigner localization problem~\eqref{WignerProblem}. We then provide an explicit solution for the localization problem for Gaussian Cohen's classes, providing new examples of quadratic time-frequency distributions whose localization problems are solved by the Hermite functions. Lastly, we connect our results on mixed-state localization operators to Toeplitz operators in quantum Gabor spaces, recently introduced in \cite{OpSTFT}.\newline

The paper is organized as follows: We collect the required preliminaries in Section~\ref{Sec:Prelims}. Section 3 gives a thorough introduction to Hagedorn wavepackets, and their representations in phase space. In Section 4 we prove several extensions of Daubechies' theorem for classical localization operators, with a particular focus on the multidimensional case. Section 5 is about Daubechies-type theorems in the general setting of mixed-state localization operators. We prove a general result which not only covers the classical Daubechies-type theorems considered in Section 4, but also covers celebrated results for the Wigner localization problem. Lastly, in Section 6 we connect our results on mixed-state localization operators to Toeplitz operators in quantum Gabor spaces.

\section{Preliminaries}\label{Sec:Prelims}
\subsection{Notation}
For a multi-index \(k\in\N_0^d\), we denote by \(|k|\) its length, \(\sum_{j=1}^d k_j\), and by \(k!\) the product \(\prod_{j=1}^d k_j!\). Lowercase letters denote functions from $\R^d$ to $\C$. The uppercase letters \(F,G,H,\dots\) will denote functions from $\R^{2d}$ to $\C$, while the uppercase letters \(R,S,T,\dots\) denote operators on $L^2(\R^d).$ Uppercase Greek letters denote functions in Bochner spaces. The Schwartz class is denoted by $\mathscr{S}(\R^d)$. We denote by $e_j$ the multi-index with a $1$ in the $j$-th spot and zeroes elsewhere. For a multi-index \(\alpha\in\N_0^d\), \(t^\alpha=t_1^{\alpha_1}t_2^{\alpha_2}\dots t_d^{\alpha_d}\). The parity operator is defined by \(Pf(t)=f(-t)\), and we will often write \(Pf\) as \(\check{f}\). 
We will frequently identify the time-frequency plane $\R^{2d}$ with the complex plane $\C^d$ through the identification \[(x_1,x_2,\dots,x_d,\omega_1,\omega_2,\dots,\omega_d)\mapsto (x_1+i\omega_1,x_2+i\omega_2,\dots, x_d+i\omega_d).\] Unless otherwise stated, \(\|\cdot\|\) always denotes the \(L^2\) norm. Inner products are always assumed to be antilinear in the second coordinate. \(f\otimes g\) denotes the rank-one operator, defined by \(\left(f\otimes g\right) (h)=\langle h,g\rangle f\). The \(2d\times 2d\)-matrix \(\begin{pmatrix}
    0 & \mathrm{Id}_d \\ -\mathrm{Id}_d & 0
\end{pmatrix}\) is denoted by $J$. A real \(2d\times 2d\)-matrix is called symplectic if \(A^TJA=J\). Lastly, \(\mathcal{F}_\sigma\) denotes the symplectic Fourier transform, defined by\begin{align*}
    \mathcal{F}_{\sigma}f(z)=\int_{\R^{2d}}f(z')e^{-2\pi i \langle Jz,z'\rangle}\;dz'.
\end{align*}

\subsection{The short-time Fourier transform}

Several concepts related to the short-time Fourier transform (STFT) play an important role in this work. In this section we present the most important properties of the STFT here, see \cite{Grochenig} for details. Given \(f,g \in L^2(\R^d)\) and \(g\neq 0\). We define the short-time Fourier transform (STFT) of \(f\) with window \(g\) as \begin{align}\label{STFTDef}
V_gf(x,\omega)=\int_{\R^d}f(t)\overline{g(t-x)}e^{-2\pi i \langle \omega, t\rangle}\;dt.
\end{align} 
If we define the time-frequency shift of \(g\) as \[\pi(z)g(t)=\pi(x,\omega)g(t)=g(t-x)e^{2\pi i \langle\omega, t\rangle},\] we may vrite $V_gf(z)=\langle f,\pi(z)g\rangle.$ A fundamental property of the STFT is Moyal's identity: \begin{prop}[{\cite[Theorem 3.2.1]{Grochenig}}]
Let $f_1,f_2,g_1,g_2\in L^2(\R^d).$ We have \begin{align}\label{Moyal}\int_{\mathbb{R}^{2d}}V_{g_1}f_1(z)\overline{V_{g_2}f_2(z)}\;dz=\langle f_1,f_2\rangle\overline{\langle g_1,g_2\rangle}.\end{align}\end{prop} As a direct consequence of Moyal's identity, we may reconstruct $f$ from $V_g f$ for any $g\neq 0$:
\begin{align*}
    f=\frac{1}{\langle g,g \rangle}\int_{\R^{2d}} V_gf(z)\pi(z)g\; dz,
\end{align*} where we interpret the integral weakly. Another consequence of Moyal's identity is that if $\|g\|=1$, then the mapping $f\mapsto V_gf$ is an isometry from $L^2(\R^d)$ to $L^2(\R^{2d}).$ Moreover, the adjoint mapping is \[V_g^* F=\int_{\R^{2d}}F(z)\pi(z) g\;dz,\] where we again interpret the integral weakly. The reconstruction formula may now equivalently be written as \begin{align*}f=V_g^*V_g f.\end{align*} The reconstruction formulas for the STFT motivate the definition of the localization operators. For a subset $\Omega\subset\R^{2d}$, and a normalized $g\in L^2(\R^d)$ we define the localization operator $A_{\Omega}^g$ by the weak integral \[A_{\Omega}^g f=\int_{\Omega} V_gf(z)\pi(z)g\; dz,\] that is, we restrict the integral in the reconstruction formula to the domain $\Omega$. In terms of operators, we may write $$A_{\Omega}^g=V_g^*m_{\Omega} V_g,$$ where \(m_{\Omega}F(z)=\chi_{\Omega}(z)\cdot F(z)\).
Localization operators are closely related to the following extremal problem: Given a domain $\Omega\subset \R^d$ and a normalized $g\in L^2(\R^d)$, find \[\sup_{\|f\|=1} \int_{\Omega} |V_gf(z)|^2\;dz.\] We call this extremal problem \textit{the spectrogram localization problem over $\Omega$.} By the Courant-Fischer theorem \cite{Lax}, the solution is given in terms of the eigenvalues of the localization operator $A_{\Omega}^g$. Namely, if $\psi_n$ is the $n$-th eigenfunction, ordered decreasingly by the corresponding eigenvalue, then 
\[\int_{\Omega} |V_g\psi_n(z)|^2\;dz=\max \bigg\{\int_{\Omega} |V_gf(z)|^2\;dz\colon \|f\|=1, f\perp\psi_1,\psi_2,\dots\psi_{n-1}\bigg\}.\] 
We also have \begin{align*}
    \int_{\Omega} |V_g\psi_n(z)|^2\;dz=\int_{\Omega} V_g\psi_n(z)\langle\pi(z) g,\psi_n\rangle\;dz=\langle A_{\Omega}^g\psi_n,\psi_n\rangle=\lambda_n\langle\psi_n,\psi_n\rangle=\lambda_n.
\end{align*} One can also study weighted extremal problems, that is, finding a function $f$ that maximizes the weighted integral~\eqref{ConcProblem} where $F$ is a ``nice" weight, a notion we will make more concrete later. Of course, letting \(F(z)=\chi_{\Omega}(z)\) recovers the spectrogram localization problem over $\Omega$. This motivates the study of the generalized localization operator \[A_F^g f=\int_{\R^{2d}} F(z)\cdot V_gf(z)\pi(z)g\;dz.\]
While the natural space to work with the STFT on is \(L^2(\R^d)\), we will sometimes make use of two other spaces, namely \textit{Feichtinger's algebra}~\cite{Fei}:\begin{align*}
    S_0(\R^d)=\big\{f\in\mathscr{S}'(\R^d)\colon V_{\phi_0}f\in L^1(\R^{2d})\big\},
\end{align*}
where \(\phi_0(t)=2^{d/4}e^{-\pi |t|^2}\) is the standard Gaussian, as well as its dual space \(S_0'(\R^d)\) of \textit{mild distributions.}

\subsection{Weyl calculus}\label{ssec:WignerWeyl}

The Weyl calculus will appear in the later parts of this work, both as a motivating example and a mathematical tool. The starting point is the Wigner distribution, which for two functions \(f,g\in L^2(\R^d)\) is defined by \begin{align}\label{WignerDef}
    W(f,g)(x,\omega)=\int_{\R^d}f(x+\frac{t}{2})\overline{f(x-\frac{t}{2})}e^{-2\pi i \langle\omega, t\rangle}\;dt.
\end{align}
In the case where \(f=g\) we write \(W(f)=W(f,g)\). The Wigner distribution is in essence an STFT, as can be seen from the formula~\cite[Lemma 4.3.1]{Grochenig}\begin{align}\label{WignerIsSTFT}
    W(f,g)(x,\omega)=2^de^{4\pi i \langle x,\omega\rangle}V_{\check{g}}f(2x,2\omega),
\end{align}
and consequently it inherits several properties from the STFT. Nevertheless, it is helpful to introduce the Wigner distribution due to its relation to two other concepts. The first one is the \textit{Weyl transform}. Given a tempered distribution \(F\) on \(\R^{2d}\) one can define an operator \(L_F\), called the Weyl transform of \(F\), via the pairing \begin{align}\label{WeylDef}
    \langle L_Ff,g\rangle=\langle F,W(g,f)\rangle,
\end{align}
where \(f\) and \(g\) are Schwartz functions on \(\R^{d}\), the left pairing is over \(\R^d\) and the right pairing is over \(\R^{2d}\). Given an operator \(T:\mathscr{S}(\R^d)\to\mathscr{S}'(\R^d)\), if there is a distribution \(F\) such that the relation \begin{align*}
    \langle Tf, g \rangle=\langle F,W(g,f)\rangle
\end{align*}
holds for all Schwartz functions \(f,g\), we call \(F\) the \textit{Weyl symbol} of \(T\), and write \(F=a_T\).

The second important concept related to the Wigner distribution is \textit{Cohen's class}. Time-frequency distributions of Cohen's class are of the form \begin{align}\label{CohenDef}
    Q(f)=Q_{\phi}(f)=W(f)*\phi, \quad \phi\in \mathscr{S}'(\R^{2d}).
\end{align}
We will also use the cross-Cohen's class, which is given by \begin{align}\label{CrossCohenDef}
    Q_{\phi}(f,g)=W(f,g)*\phi.
\end{align}
Cohen's class contains several important time-frequency distributions, most notably the Wigner distribution itself (\(\phi=\delta\)) and the spectrograms \(|V_gf(z)|^2\) (\(\phi=W(g)\)). In fact, Cohen's class contains all covariant and separately weak*-continuous quadratic time-frequency distributions, see~\cite[Theorem 4.5.1]{Grochenig} or~\cite{GSMetaplectic}.
\subsection{Special functions}
We will make frequent use of three classes of special functions, which we will introduce below.\newline

The Hermite functions, defined by \begin{align}\label{HermiteDef}
    \phi_n(t)=\frac{2^{1/4}}{\sqrt{n!}}\left(\frac{-1}{2\sqrt{\pi}}\right)^ne^{\pi t^2}\frac{d^n}{dt^n}\left(e^{-2\pi t^2}\right),
\end{align} 
play an important role in time-frequency analysis and many related areas of mathematics. They are eigenfunctions of the harmonic oscillator, they form an orthonormal basis of \(L^2(\R)\), and they are known to satisfy a three-term recurrence relation. By the three-term recurrence, as well as the properties of the Gaussian function \(\phi_0(t)=2^{1/4}e^{-\pi t^2}\), it is clear that we also have the formula \[\phi_n(t)=h_n(t)\phi_0(t),\] where \(h_n\) is a polynomial of degree \(n\), called the \(n\)-th Hermite polynomial.

In higher dimensions we will also consider the product Hermite functions, defined by \[\phi_{k}(t_1,t_2,\dots,t_d)=\prod_{j=1}^d \phi_{k_j}(t_j),\quad k\in\N_0^d.\] Most properties of the Hermite functions carry over to higher dimensions.

Another class of special functions we need is the generalized Laguerre polynomials. For \(k\in\N_0\), \(\alpha,t>0\) we define \begin{align*}
    L_k^\alpha(t)=\sum_{j=0}^k(-1)^j\binom{k+\alpha}{k-j}\frac{t^j}{j!}.
\end{align*}
The Laguerre polynomials satisfy the recurrence relation \[L_{k+1}^\alpha(t)=\frac{2k+1+\alpha-t}{k+1}L_{k}^\alpha(t)-\frac{k+\alpha}{k+1}L_{k-1}^\alpha(t),\] where \(L_0^\alpha(t)=1\) and \(L_1^\alpha(t)=1+\alpha-t\). For any fixed \(\alpha\), the Laguerre polynomials are orthonormal on the weighted \(L^2\)-space \(L^2(\R_+,t^\alpha e^{-t} \; dt)\).

The last set of special functions we need is the complex Hermite polynomials, a class of polynomials over \(\C\) defined via the Laguerre polynomials. Given \(z\in\C, n,k\in\N_0\), we define \begin{align*}H_{n,k}(z)=\begin{cases}
    \sqrt{\frac{k!}{n!}}\pi^{\frac{n-k}{2}}z^{n-k}L_k^{n-k}(\pi|z|^2),\quad 0\leq k<n \\ (-1)^{k-n}\sqrt{\frac{n!}{k!}}\pi^{\frac{k-n}{2}}\overline{z}^{k-n}L_n^{k-n}(\pi|z|^2),\quad 0\leq n\leq k.
\end{cases}\end{align*}
Since the complex Hermite polynomials have different definitions depending on the relation between \(n\) and \(k\) we would in general have to consider several different cases in our proofs. However, since the Laguerre polynomials satisfy the reflection identity \begin{align}\label{Reflection}
    \frac{(-t)^n}{n!}L_{j}^{n-j}(t)=\frac{(-t)^j}{j!}L_{n}^{j-n}(t),
\end{align} we will usually only consider the first case, even if \( n-j<0\), and interpret it in terms of the reflection identity.\newline 

It is well-known that the complex Hermite polynomials are orthogonal with respect to the Gaussian measure, that is $$\int_{\mathbb{C}}H_{n,k}(z)\overline{H_{m,j}(z)}e^{-\pi |z|^2}\;dz=\delta_{n,m}\delta_{j,k}.$$ In \cite{LargeSieve}, it was shown that the complex Hermite polynomials are also locally orthogonal. We will make use of this result in the special case $k=j$. The statement below is slightly different from the one made \cite{LargeSieve}, as we have corrected the expression for the constant.
\begin{lemma}[{\cite[Prop. 7]{LargeSieve}}]\label{AbreuSpeckbacher}
    Let \(n,m,k\in\N_0\) and let $D_R$ be the disc of radius $R$ centered at $0$.  We have 
    \begin{align}\label{HermiteDblOrth}
        \int_{D_R} H_{n,k}(z)\overline{H_{m,k}(z)}e^{-\pi |z|^2}\; dz=c_{n,k}(R)\delta_{n,m},\end{align} with \begin{align*}
            c_{n,k}(R)=\begin{cases}
        \frac{k!}{n!}\int_0^{\pi R^2}t^{n-k}\left(L_k^{n-k}(t)\right)^2e^{-t}\; dt\quad 0\leq k< n, \\
        \frac{n!}{k!}\int_0^{\pi R^2}t^{k-n}\left(L_n^{k-n}(t)\right)^2e^{-t}\; dt\quad 0\leq n\leq k.
    \end{cases}
    \end{align*}
\end{lemma}
\begin{remark}
    Observe that the case $k=0$ gives the eigenvalues from Daubechies' classical theorem.
\end{remark}

With regards to time-frequency analysis, the importance of the complex Hermite polynomials stems from the Laguerre connection \cite[Section 1.9]{Folland}.

\begin{lemma}[Laguerre connection]
    The following formula for the STFT of a Hermite function with respect to another Hermite function holds.
    \begin{align}\label{LaguerreConn} V_{\phi_k}\phi_n(x,\omega)=e^{-i\pi x\omega}e^{-\frac{\pi}{2}(x^2+\omega^2)}\overline{H_{n,k}(z)}.\end{align}
\end{lemma}
For convenience, we will sometimes use the following polar decomposition of the STFT of a Hermite function with Hermite window. If \(k<n\), we write \begin{align}\label{Decomp}
    V_{\phi_k}\phi_n(z)=\rho_{n,k}(r)e^{i(k-n)\theta}e^{-\pi i xw},
\end{align} where  \(\rho_{n,k}(r)=\sqrt{\frac{k!}{n!}\pi^{n-k}}r^{n-k}L_k^{n-k}(\pi r^2)e^{-\pi r^2/2}\) is the radial function such that the right hand side agrees with the Laguerre connection. If \(n\leq k\) we write \[V_{\phi_k}\phi_n(z)=\rho_{n,k}(r)e^{i(n-k)\theta}e^{-i\pi xw},\] again with the same interpretation. By tensorization the same formulas carry over to the multidimensional case.

\subsection{Operator theory}

Some places in this work we will need the Schatten classes of operators, which are the operator analogues of the \(L^p\)-spaces. They are defined using the singular value decomposition of operators, as stated below~\cite{ReedSimon}.
\begin{prop}
    Let \(S\in\mathcal{B}(L^2(\R^d))\) be a compact operator. There are two orthonormal sets \(\{\psi_n\}_{n\in\N}\), \(\{\phi_n\}_{n\in\N}\) in \(L^2(\R^d)\) and a sequence \(\{\sigma_n\}_{n\in\N}\) of non-negative numbers converging to \(0\) such that \begin{align*}
        S=\sum_{n=0}^{\infty} \sigma_n\psi_n\otimes\phi_n
    \end{align*}
    holds in the operator norm. 
\end{prop}

For \(p\in[1,\infty)\) the \textit{Schatten class} \(\mathcal{S}^p\) of operators is defined by \begin{align*}
    \mathcal{S}^p=\Big\{S\in\mathcal{B}(L^2(\R^d))\colon S\text{ compact, }\{\sigma_n\}\in\ell^p\Big\}.
\end{align*}
With the norm \(\|S\|_{\mathcal{S}^p}=\|\{\sigma_n\}\|_{\ell^p}\), each Schatten class becomes a Banach space. We also define \(\mathcal{S}^\infty=\mathcal{B}(L^2(\R^d)).\) For \(p,q\in[1,\infty]\) we have \begin{align*}
    p\leq q\implies\|S\|_{\mathcal{S}^q}\leq \|S\|_{\mathcal{S}^p}.
\end{align*}

The smallest Schatten class \(\mathcal{S}^1\) is often referred to as the \textit{trace class operators}. For a positive operator \(S\), the trace is defined to be \begin{align*}
    \mathrm{tr}(S)=\sum_{n\in\N} \langle S e_n,e_n\rangle,
\end{align*}
where \(\{e_n\}_{n\in\N}\) is any orthonormal basis of \(L^2(\R^d)\). It can be shown that the value of the trace is independent of the choice of basis. If \(S\in\mathcal{S}^1\), then the trace is well-defined and \(\mathrm{tr}(S)=\sum_{n\in\N} \sigma_n\). Likewise, if we define \(|S|^p\) via functional calculus then we also have \(\|S\|_{\mathcal{S}^p}^p=\mathrm{tr}(|S|^p)\). The class \(\mathcal{S}^2\) is called the \textit{Hilbert-Schmidt operators}, and is often denoted by \(\mathcal{HS}\). It is a Hilbert space under the inner product \begin{align*}
    \langle S,T\rangle_{\mathcal{HS}}=\mathrm{tr}(ST^*).
\end{align*}

\subsection{Quantum harmonic analysis}

In this section we will recall some of the fundamental notions in quantum harmonic analysis, introduced by Werner in \cite{Werner}. The idea underlying it all is to introduce operator analogues of classical function theoretic notions, and then study how they interact with the classical ones. For further detail we direct the reader to \cite{bible1}. First we will define convolutions, for which we will need the operator shift. Given \(z\in\R^{2d}\) and \(S\in B(L^2(\R^d))\) we define the shifted operator \(\alpha_z(S)\) by \begin{align*}
    \alpha_z(S)=\pi(z)S\pi(z)^*.
\end{align*}

Using the properties of the time-frequency shift, one can easily check that \(\alpha_{z}\alpha_{w}(S)=\alpha_{z+w}(S)=\alpha_{w}\alpha_{z}(S)\) and that \(\|\alpha_z(S)\|_{\mathrm{op}}=\|S\|_{\mathrm{op}}\), so one should think of \(\alpha\) as an operator analogue of the translation operator. We also have an operator analogue of reflection. Namely, if we let \(P\) denote the parity operator, \(Pf(t)=f(-t)\), we define the operator \(\check{S}\) by \begin{align*}
    \check{S}=PSP.
\end{align*} 
Using these two operations, Werner defined two convolutions. Given a function \(f\in L^1(\R^{2d})\) and an operator \(S\in\mathcal{S}^1(\R^d)\) we define the trace class operator \(f\star S\) by \begin{align}\label{FOConv}
    f\star S=\int_{\R^{2d}}f(z)\alpha_z(S)\;dz,
\end{align}
where we interpret the integral weakly. For two operators \(S,T\in\mathcal{S}^1(\R^d)\) we define the \(L^1\)-function \(S\star T\) by \begin{align}\label{OOConv}
    S\star T(z)=\mathrm{tr}\left(S\alpha_z\left(\check{T}\right)\right).
\end{align}
While initially only defined for \(L^1\)-functions and trace class operators, the convolutions can be extended to other spaces. By duality, the convolutions extend to the other Lebesgue spaces and Schatten classes, for which we also have an analogue of Young's inequality: If \(p,q,r\in[1,\infty]\) such that \(\frac{1}{p}+\frac{1}{q}=1+\frac{1}{r}\) and we have \(f\in L^p(\R^{2d})\), \(S\in \mathcal{S}^p\) and \(T\in\mathcal{S}^q\), then the convolutions satisfy \begin{align*}
    \|f\star T\|_{\mathcal{S}^r}\leq \|f\|_{L^p}\|T\|_{\mathcal{S}^q}, \\
    \|S\star T\|_{L^r}\leq\|S\|_{\mathcal{S}^p}\|T\|_{\mathcal{S}^q}. 
\end{align*}
Young's inequality ensures that both convolutions are separately continuous in each variable. The most general extension of the convolutions we will need in this work is to tempered distributions and Schwartz operators. It was shown in \cite{SchwartzOps} that the convolutions extend to these classes, and that they also are separately continuous in each variable. The quantum convolutions inherit several properties of regular convolution. For example, they are bilinear, commutative and associative. The associativity is particularly striking, as it shows the interplay between the two quantum convolutions and the classical one. \newline

For illustrative purposes, we mention that in the case where \(S\) and \(T\) are rank-one operators \(f\otimes f\) and \(g\otimes g\), the convolutions are familiar objects from time-frequency analysis. Given an \(F\in L^1(\R^{2d})\), we have that \begin{align}\label{LocOpIsConv}
    F\star (g\otimes g)=A_F^g \end{align} and \begin{align}\label{SpecIsConv}
    (f\otimes f)\star(\check{g}\otimes \check{g})(z)=|V_{g}f(z)|^2.
\end{align}
Replacing the rightmost rank-one operator with a general one, we recover another familiar object, namely the Cohen's class distributions:\begin{align}\label{CCIsConv}
    f\otimes f\star \check{A}(z)=Q_{a_A}(f)(z),
\end{align}
where \(a_A\) is the Weyl symbol of \(A\). Likewise, a cross-Cohen's class distribution \(Q_{a_A}(f,g)\) can be written as \(Q_{a_A}(f,g)=f\otimes g\star \check{A}.\) Following \cite{bible2}, we will often describe a Cohen's class distribution as an operator convolution, and write \(Q_A(f)=Q_{a_A}(f)\). Similarly, we will write \(Q_A(f,g)=(f\otimes g)\star \check{A}.\)

In quantum mechanics, positive operators with trace \(1\) are called states, and they form a convex subset of \(\mathcal{B}(L^2(\R^d))\). The extreme points of this set are called pure states, and it can be shown that the pure states are the rank-one operators \(f\otimes f\) with \(\|f\|=1\)~\cite[Chapter 13]{deGossonBook}. With this in mind, if \(S\) is a state we will often refer to the function-operator convolution \(F\star S\) as a \textit{mixed-state localization operator}. In fact, we will interpret any function-operator convolution as localizing some time-frequency distribution.

In quantum harmonic analysis, there is also an operator analogue of the Fourier transform. For a trace class operator \(S\in \mathcal{S}^1\) we define the Fourier-Wigner transform of \(S\) as the function \begin{align}\label{FWDef}
    \mathcal{F}_W(S)(z)=e^{-\pi i \langle x,\omega\rangle}\mathrm{tr}\left(\pi(-z)S\right).
\end{align}
As expected, the Fourier-Wigner transform can be extended to Schwartz operators, and the result will be a tempered distribution. The Fourier-Wigner transform is invertible, and its inverse can be written in terms of the Weyl transform and the symplectic Fourier transform. In fact, we have
\begin{align*}
    \mathcal{F}_\sigma\mathcal{F}_W(L_f)=f, \qquad
    L_{\mathcal{F}_\sigma\mathcal{F}_W(T)}=T.
\end{align*}
The Fourier-Wigner transform of a rank-one operator is another familiar object, namely the ambiguity function: \begin{align*}
    \mathcal{F}_W(f\otimes f)(z)=e^{\pi i \langle x,\omega\rangle}V_{f}f(z).
\end{align*}

\subsection{Notions from quantum time-frequency analysis}

The authors of \cite{OpSTFT} introduced a short-time Fourier transform of operators as a tool to do time-frequency analysis of operators. Naturally, there is a connection to classical time-frequency analysis of functions, but also to quantum harmonic analysis, both of which we will make use of in what follows. 

\begin{definition}
    Let \(S,T\in \mathcal{S}^2\). The \textit{short-time Fourier transform} of \(T\) with respect to \(S\) is \begin{align*}
        \mathfrak{V}_ST(z)=S^*\pi(z)^*T.
    \end{align*}
\end{definition}
Unlike the classical STFT, which is a scalar-valued function, the operator STFT is operator-valued. In particular, it can be shown that \(\mathfrak{V}_ST\) is in the Bochner-Lebesgue space \( L^2(\R^{2d};\mathcal{S}^2)\). The operator STFT inherits several properties from the classical STFT, such as a Moyal identity and a reconstruction property. For details, we refer to \cite{OpSTFT}. The main property we will need is that its image in the Bochner-Lebesgue space is a reproducing kernel Hilbert space.
\begin{lemma}
    Let \(S\in\mathcal{S}^2\). The set \begin{align*}
        \mathfrak{V}_S(\mathcal{S}^2)=\{\mathfrak{V}_ST\colon T\in\mathcal{S}^2\}
    \end{align*}
    is a uniform reproducing kernel Hilbert space. The reproducing kernel is \begin{align*}
        K_S(z,z')=S^*\pi(z)^*\pi(z')S.
    \end{align*}
\end{lemma}
We anticipate the connection to quantum harmonic analysis and call the space \( \mathfrak{V}_S(\mathcal{S}^2)\) the \textit{quantum Gabor space}. Because of the reproducing kernel structure we may define a projection from \( L^2(\R^{2d};\mathcal{S}^2)\) to \(\mathfrak{V}_S(\mathcal{S}^2)\) via the formula \begin{align}\label{ProjFormula}
    P_S\Psi(z)=\int_{\R^{2d}} K_S(z,z')\Psi(z')\;dz',
\end{align}
which equals \(\mathfrak{V}_S\mathfrak{V}_S^*\Psi\). The authors of \cite{OpSTFT} also introduced Toeplitz operators on the quantum Gabor space. Given a mask \(F\in L^{\infty}(\R^{2d})\) we define the Toeplitz operator on \(\mathfrak{V}_S(\mathcal{S}^2)\) with mask \(F\) by \begin{align}\label{ToeplitzDef}
    T_F(\Psi)=P_S(F\cdot\Psi),
\end{align}
for \(\Psi\in \mathfrak{V}_S(\mathcal{S}^2)\). If the operator is sufficiently nice, we may also define Toeplitz operators with masks in \(\mathscr{S}'(\R^{2d})\) by interpreting everything in the sense of distributions. Just as the Gabor-Toeplitz operators are unitarily equivalent to localization operators with the same mask, it was shown that the Toeplitz operator \(T_F\) is unitarily equivalent to the operator \begin{align*}
    R\mapsto (F\star (SS^*))\circ R
\end{align*}
on \(\mathcal{S}^2\) via the conjugation operator \(\Theta(T_F)=\mathfrak{V}_S^*T_F\mathfrak{V}_S.\) So the quantum Gabor-Toeplitz operators are unitarily equivalent to composing with a mixed-state localization operator with the same mask. 

\subsection{Reinhardt domains}
Here we present some fundamental properties and examples of Reinhardt domains, for further details, see \cite{SCVBook}.
Reinhardt domains are subsets of \(\C^d\) that can be realized as a ``complex solid of revolution." We make this analogy precise by introducing the following two objects.
\begin{definition}
        We call the set \[V^d=\{(r_1,r_2,\dots ,r_d)\in \mathbb{R}^d\colon r_i\geq 0 \;\forall \;i\in\{1,2,\dots,d\}\}\] \textit{the absolute space}. The natural projection of $\C^{d}$ onto $V^d$ is \[\tau(z_1,z_2,\dots, z_d)=(|z_1|,|z_2|,\dots, |z_n).\]
\end{definition}
Any subset of \(\C^d\) can be projected onto \(V^d\). The defining property of Reinhardt domains is that they can be recovered from their projected image in absolute space. In the following definition \(\tau^{-1}(W)\) will denote the pre-image of \(W\) under \(\tau\).
\begin{definition}
    Let \(\Omega\subset \C^d\) be open. \(\Omega\) is a \textit{Reinhardt domain} if for all \(z\in \Omega\), we have \(\tau^{-1}(\{\tau(z)\})\subset \Omega.\)
\end{definition}
Any open subset \(W\subset V^d\) defines a Reinhardt domain, namely \(\tau^{-1}(W)\). The following lemma, which is Theorem \(1.2\) in \cite{SCVBook}, shows that these are in fact all Reinhardt domains, justifying their interpretation as solids of revolution.
\begin{lemma}\label{Lem:Reinhardt}
    An open set $\Omega\subset \C^d$ is a Reinhardt domain if and only if there is an open set $W\subset V^d$ such that $\Omega=\tau^{-1}(W).$
\end{lemma}
The set \(W\) in the above lemma is referred to as the \textit{Reinhardt shadow} of \(\Omega\) \cite{ReinhardtArticle}. As solids of revolution, all Reinhardt domains \(\Omega=\tau^{-1}(W)\) admit a nice representation in polycylindrical coordinates: \begin{align}\label{eq:Polydecomp}
    \Omega=\Bigl\{(r_1e^{i\theta_1},\dots,r_de^{i\theta_d})\in\C^d\colon (r_1,\dots,r_d)\in W, (\theta_1,\dots,\theta_d)\in \T^d\Bigr\},
\end{align}
where \(\T=[0,2\pi]\). Many familiar domains are Reinhardt domains. We collect some of them in the example below. For more exotic examples, see~\cite{ReinhardtExamples}. \begin{exmp}\label{Ex:Reinhardts}
    \begin{itemize}
        \item The ball in \(\C^d\) with radius \(R\) is a Reinhardt domain. Its shadow is \(\{r\in V^d\colon \sum_{j=1}^d r_j^2\leq R^2\}\).
        \item The polydisc in \(\C^d\) with polyradius \(R=(R_1,R_2,\dots, R_d)\) is a Reinhardt domain. Its shadow is \(\prod_{j=1}^d [0,R_j]\subset V^d\).
        \item In general, any \(p\)-norm ball centered at \(0\) will be a Reinhardt domain. The shadow is \(\{r\in V^d\colon \sum_{j=1}^d r_j^p\leq R^p\}\). Note that this is true even if \(p\in (0,1)\), although these domains will not be convex.
        \item Given a multi-index \(\alpha\in \N^d\) the set \(\{z\in\C^d\colon |z^\alpha|^2\leq R\}\) is also a Reinhardt domain. The shadow is \(\{r\in V^d\colon \sum_{j=1}^d r_j^{2\alpha_j}\leq R\}\). Note that this Reinhardt shadow is not symmetric along the line \(r_1=r_2=\dots =r_d\).
    \end{itemize}
\end{exmp}

\section{Hagedorn wavepackets}\label{Sec:Hagedorn}

Hagedorn wavepackets are a generalization of the product Hermite functions that allows for more flexibility. For instance, the variables of a Hagedorn wavepacket may depend on each other, unlike the Hermite functions where all variables are independent. We will base our treatment of Hagedorn wavepackets on~\cite{Hagedorn1,Hagedorn2}. Their construction depends on two parameters: two complex matrices \(Q,P\in\C^{d\times d}\) forming a \textit{Lagrangian frame}. A Lagrangian frame must satisfy several conditions. Namely, \(Q\) and \(P\) must be invertible, and satisfy 
\begin{align}\label{Lagrange}
    Q^TP-P^TQ=0,\quad Q^*P-P^*Q=2i\mathrm{Id}_d.
\end{align}
Condition (\ref{Lagrange}) is equivalent to 
\begin{align}\label{Symplectic}
    T=\begin{pmatrix}
        \mathrm{Re}(Q) & \mathrm{Im}(Q) \\ \mathrm{Re}(P) & \mathrm{Im}(P)
    \end{pmatrix} \text{ is symplectic.}
\end{align}
Thus, we may treat a Lagrangian frame either as two invertible matrices \(Q,P\in \C^{d\times d}\), or as the symplectic matrix \(T\in\R^{2d\times 2d}\). The product Hermite functions are defined by repeated use of the ladder operators \(A_k^\dagger\) defined by \begin{align*}
    A_k^\dagger f(t)=A_k^\dagger f(t_1,t_2,\dots, t_d)= \sqrt{\pi}t_kf(t)-\frac{1}{2\sqrt{\pi}}\frac{\partial f}{\partial t_k}(t),
\end{align*}
on the \(d\)-dimensional Gaussian \(\phi_0(t)=2^{\frac{d}{4}}e^{-\pi |t|^2}\), see e.g. \cite{Folland}. The Hagedorn wavepackets are constructed in much the same way. Given \( Q, P\) as above we consider the Gaussian wavepacket \begin{align*}
    \phi_0[Q,P](t)=2^{\frac{d}{4}}\mathrm{det}(Q)^{-\frac{1}{2}}e^{\pi i\langle t,PQ^{-1}t\rangle},\quad t\in\mathbb{R}^d.
\end{align*}
From condition (\ref{Lagrange}) we may deduce that \(\mathrm{Im}(PQ^{-1})=(QQ^{*})^{-1}>0\) \cite{Hagedorn1}, and \(\phi_0[Q,P]\) is thus square integrable with norm \(1\). The parameters now also get a physical interpretation:  \(QQ^*\) is the covariance matrix of the position density \(|\phi_0[Q,P](t)|^2\), while \(PP^*\) is the covariance of the momentum density \(|\mathcal{F}(\phi_0[Q,P])(\xi)|^2\).\newline

The higher-order Hagedorn wavepackets will be generated by use of the generalized ladder operator given in vector form by\begin{align*}
    A^\dagger[Q,P]=\sqrt{\pi}i\left(P^*\cdot \mathbf{t}+Q^*\left(\frac{i\nabla}{2\pi}\right)\right).
\end{align*}
The practical interpretation is given by the lemma below, which collects several important properties of the ladder operator. A proof can be found in \cite{Hagedorn1}.
\begin{lemma}
    Let \(Q,P\) satisfy (\ref{Lagrange}). Let \(A^\dagger\) denote the ladder operator \(A^\dagger[Q,P]\). The \(k\)-th component of the ladder operator acts on a Schwartz function \(f\) by \begin{align*}
        A_k^\dagger f(t)=\sqrt{\pi}i\left(\sum_{l=1}^d  \overline{p}_{lk}t_lf(t)+\frac{i}{2\pi}\overline{q}_{lk}\frac{\partial f}{\partial t_l}(t)\right).
    \end{align*}
    Moreover, the components of \(A^\dagger\) commute, so we have \begin{align*}
        A_k^\dagger A_j^\dagger=A_j^\dagger A_k^\dagger.
    \end{align*} We may therefore define \(A^\dagger f\) as \begin{align*}
        A^\dagger f(t)=A_1^\dagger\left(A_2^\dagger \left(\dots A_d^\dagger f(t)\right)\right).
    \end{align*}
    Its adjoint operator is \begin{align*}
        A[Q,P]=-\sqrt{\pi}i\left(P^T\cdot \mathbf{t}+Q^T\left(\frac{i\nabla}{2\pi} \right)\right),
    \end{align*}
    interpreted in the same component-wise way.
\end{lemma}
Using the above lemma we can now define the higher-order Hagedorn wavepackets. Note that for a multi-index \(k\in\N_0^d\) we define \(\left(A^\dagger\right)^k=(A_1^\dagger)^{k_1}(A_2^\dagger)^{k_2}\dots (A_d^\dagger)^{k_d} \), which is well-defined, since the components commute.
\begin{definition}
    Let \(Q,P\) satisfy (\ref{Lagrange}), and let \(k\in \N_0^d\). The \textit{\(k\)-th Hagedorn wavepacket} is given by \begin{align*}
        \phi_k[Q,P]=\frac{1}{\sqrt{k!}}\left(A^\dagger[Q,P]\right)^k\phi_0[Q,P].
    \end{align*} 
    Using the adjoint operator lowers the order of a Hagedorn wavepacket. If $j\in \{1,\dots d\}$ then \begin{align*}
        A_j[Q,P]\phi_k[Q,P]=\sqrt{k_j}\phi_{k-e_j}[Q,P].
    \end{align*}
\end{definition}

Observe that if \(Q=\mathrm{Id}_d, P=i\mathrm{Id}_d\), then $\phi_0[Q,P]$ is the usual \(d\)-dimensional Gaussian, and the ladder operator is also the usual ladder operator in $d$ dimensions. Consequently, the wavepackets \(\phi_k[\mathrm{Id}_d,i\mathrm{Id}_d]\) are the usual product Hermite functions. More generally, if \(Q\in\R^{d\times d}, iP\in\R^{d\times d}\) with $Q$ and $P$ symmetric, then the Hagedorn wavepackets are rotated and scaled Hermite functions \cite{HagedornCite}. The strength of Hagedorn wavepackets is that one can pick any two complex matrices $Q,P$ satisfying (\ref{Lagrange}), and still inherit several properties from the Hermite functions. We collect the ones most important to us below. For the proofs, as well as further properties, we direct the reader to \cite{HagedornSurvey}.

\begin{lemma}
    Let \(k\in\N_0^d\) and let \(Q,P\) satisfy (\ref{Lagrange}). The \(k\)-th Hagedorn wavepacket can be written as \begin{align*}
        \phi_k[Q,P](t)=\frac{1}{\sqrt{2^{|k|}k!}}p_k(t)\phi_0[Q,P](t),
    \end{align*}
    where \(p_k\) is a polynomial of degree \(|k|\) satisfying the three-term recurrence relation 
    \[\left(p_{k+e_j}(t)\right)_{j=1}^d=2\sqrt{2\pi}(Q^{-1}t)p_k(t)-2Q^{-1}\overline{Q}\left(k_jp_{k-e_j}(t)\right)_{j=1}^d.\]
    The polynomial prefactor also satisfies the parity relation \begin{equation}\label{eq:parity}
        p_k(-t)=(-1)^{|k|}p_k(t).
    \end{equation}
\end{lemma}

\begin{remark}
    In \cite{Hagedorn1} the authors consider Hagedorn wavepackets with two additional parameters $q$ and $p$, corresponding to the center of the wavepacket $\phi_0.$ One can similarly define off-centered wavepackets in our setup by using the ladder operator on the wavepacket \(\pi(z)\phi_0[Q,P]\). However, in our setup the ladder operator $A^\dagger$ commutes with all time-frequency shifts. Therefore every statement above may be translated to off-centered wavepackets simply by applying a time-frequency shift. From now on we will therefore only consider wavepackets centered at $0$.
\end{remark}

\subsection{Some special two-dimensional wavepackets}\label{HagedornEx}

As seen in \cite{Hagedorn3}, the Hagedorn wavepackets in one dimension are simply rotated and scaled Hermite functions. One might therefore expect this to also be the case in higher dimensions, but this is not true. Hagedorn wavepackets only directly relate to the Hermite functions in the special cases \(Q,iP\in\R^{d\times d}\). However, we can still find explicit formulas even when the wavepackets aren't directly related to the Hermite functions. As an example, consider \(d=2\). We will look at the wavepackets with \(Q,iP\notin\R^{2\times 2}\) that are the closest to the Hermite functions. As noted in \cite{Hagedorn1}, whereas the generalized Gaussian \(\phi_0[Q,P]\) depends on both \(Q\) and \(P\), the polynomial factor of the wavepackets only depends on the matrix \(Q^{-1}\overline{Q}=M\). Following \cite{Hagedorn3}, it can be shown that if \begin{align*}
    M=\begin{pmatrix}
        0 & e^{i\theta}\\ e^{i\theta} & 0
    \end{pmatrix},\quad \theta\in[0,2\pi],
\end{align*}
then the polynomial prefactors of the higher-order wavepackets are \begin{align}\label{0Dprefactor}
    \Tilde{p}_{n}(t)=\begin{cases}
        (-1)^{n_2}e^{i\theta n_2} (n_2)!t_1^{n_1-n_2}L_{n_2}^{n_1-n_2}\left(t_1t_2e^{-i\theta}\right), \; n_1\geq n_2, \cr
        (-1)^{n_1}e^{i\theta n_1} (n_1)!t_2^{n_2-n_1}L_{n_1}^{n_2-n_1}\left(t_1t_2e^{-i\theta}\right), \; n_2\geq n_1.
    \end{cases}
\end{align}
These polynomials are much more in line with the complex Hermite polynomials than the regular product Hermite polynomials.
In this case the Hagedorn wavepackets can now be written explicitly as \begin{align}\label{explicitly}
    \phi_n[Q,P](t)=\frac{1}{\sqrt{n_1!n_2!}}\Tilde{p}_n\left(2\sqrt{\pi}Q^{-1}\begin{pmatrix}
        t_1\\ t_2
    \end{pmatrix}\right)\phi_0[Q,P](t).
\end{align}
The matrices \(Q\) that produces wavepackets like this are of the form \begin{align*}
    Q=\begin{pmatrix}
        q_1 & q_1e^{-i\theta}\\ q_2 e^{-i\theta} & q_2
    \end{pmatrix},\quad q_1,q_2\neq 0.
\end{align*}
The corresponding \(P\) can be found using Equation \eqref{Lagrange}, but note that there are several \(P\in\C^{d\times d}\) that makes \(Q,P\) a normalized Lagrangian frame. In general we will call the Hagedorn wavepackets with \(Q^{-1}\overline{Q}=\begin{pmatrix}
    0 & e^{i\theta} \\ e^{i\theta} & 0
\end{pmatrix}\) \textit{zero-diagonal wavepackets}. As a very explicit example, let us chose \(Q=\begin{pmatrix}
    1 & \frac{1-i}{\sqrt{2}} \\ \frac{1-i}{\sqrt{2}} & 1
\end{pmatrix}\), that is \(q_1=q_2=1, \theta=\frac{\pi}{4}\). Then one possible \(P\) is \begin{align*}
    P=\begin{pmatrix}
        i-1 & 0 \\
        0 & i-1
    \end{pmatrix}.
    \end{align*}
The properties of zero-diagonal wavepackets will in general depend on both \(Q\) and \(P\), but note that their absolute value only depends on \(Q\). Firstly, the covariance of the generalized Gaussian \(\phi_0[Q,P]\) is \begin{align*}
    QQ^*=\begin{pmatrix}
        |q_1|^2+|q_2|^2 & 2\overline{q_1q_2}\cos(\theta) \\2\overline{q_1q_2}\cos(\theta) &|q_1|^2+|q_2|^2
    \end{pmatrix}.
\end{align*}
The polynomial prefactors are skewed by a factor \(Q^{-1}\), which equals \begin{align*}
    Q^{-1}=\frac{1}{q_1q_2(1-e^{-i\theta})}\begin{pmatrix}
        q_2 & -q_1e^{-i\theta} \\ -q_2e^{-i\theta} & q_1.
            \end{pmatrix}
\end{align*}
Thus, independently of \(P\), we may write the absolute values explicitly using Equation \eqref{0Dprefactor} and \eqref{explicitly}. If we also know \(P\) we get a fully explicit description of \(\phi_k[Q,P]\). Returning to the case where \(Q=\begin{pmatrix}
    1 & \frac{1-i}{\sqrt{2}} \\ \frac{1-i}{\sqrt{2}} & 1
\end{pmatrix}, P=\begin{pmatrix}
        i-1 & 0 \\
        0 & i-1
    \end{pmatrix}\)  we have \begin{align*}
            \left(QQ^*\right)^{-1}=\begin{pmatrix}
        1 & -\frac{1}{\sqrt{2}} \\ -\frac{1}{\sqrt{2}} & 1
    \end{pmatrix}
    \quad\text{and}\quad 
        Q^{-1}=\frac{1}{2}\begin{pmatrix}
            1-i & \sqrt{2}i \\ \sqrt{2}i & 1-i
        \end{pmatrix}.
    \end{align*}
    Using \eqref{explicitly}, the wavepackets are 
    \begin{multline}\label{0Dformula}
    \phi_{n_1,n_2}[Q,P](t_1,t_2)= \\\begin{cases}
            2^{1/4}\left(-\frac{1+i}{\sqrt{2}}\right)^{n_2}\frac{\sqrt{\pi}^{n_1-n_2}\sqrt{n_2!}}{\sqrt{n_1!}}\left(t_1+i(\sqrt{2}t_2-t_1)\right)^{n_1-n_2}L_{n_2}^{n_1-n_2}\left(2\sqrt{\pi}\xi\right)e^{-\pi\xi} e^{-\frac{i}{\sqrt{2}}\pi t_1t_2},n_1\geq n_2,
            \cr
            2^{1/4}\left(-\frac{1+i}{\sqrt{2}}\right)^{n_1}\frac{\sqrt{\pi}^{n_2-n_1}\sqrt{n_1!}}{\sqrt{n_2!}}\left(t_2+i(\sqrt{2}t_1-t_2)\right)^{n_2-n_1}L_{n_1}^{n_2-n_1}\left(2\sqrt{\pi}\xi\right)e^{-\pi\xi}e^{-\frac{i}{\sqrt{2}}\pi t_1t_2}, n_2\geq n_1,
        \end{cases}
\end{multline}
where \(\xi(t_1,t_2)=t_1^2-\sqrt{2}t_1t_2+t_2^2=\langle \left(QQ^*\right)^{-1}t,t\rangle\).)

\subsection{Hagedorn wavepackets in phase space}
In order to examine localization operators with a Hagedorn wavepacket window, we need information about the STFT of Hagedorn wavepackets. Firstly, the cross-Wigner distribution of two Hagedorn wavepackets was treated in \cite{Hagedorn1}, where the following formula was found. We have slightly modified the formula to fit with our conventions.
\begin{lemma}\label{Wigner}
    Let $k,n\in\mathbb{N}_0^d,$ let \(Q,P\) satisfy \eqref{Lagrange}, and let $T$ be the corresponding symplectic matrix. The cross-Wigner distribution of \(\phi_n[Q,P]\) and \(\phi_k[Q,P]\), denoted by \(W(n,k)\), is 
    \begin{align*}
        W(n,k)(x,\omega)=2^{d} (-1)^{|k|}e^{-2\pi|\zeta|^2}\prod_{j=1}^d\overline{H_{n_j,k_j}(\sqrt{2}\zeta_j)},
    \end{align*}
    where \(\zeta=\zeta(x,\omega)=-iP^Tx+iQ^T\omega=T^{-1}\begin{pmatrix}
        x \\ \omega
    \end{pmatrix},\) and \(H_{n_j,k_j}\) is the corresponding complex Hermite polynomial.
\end{lemma}
This, together with Equation~\eqref{WignerIsSTFT} gives a formula for the STFT of a Hagedorn wavepacket with a Hagedorn wavepacket window.
\begin{prop}\label{HagedornSTFT}
        Let $k,n\in\mathbb{N}_0^d$, \(Q,P\) and $T$ be as in Lemma \ref{Wigner}. If we write $V_{k}n(x,\omega)=V_{\phi_k[Q,P]}\phi_n[Q,P](x,\omega)$, then we have \begin{align}\label{HagedornSTFTFormula}
        V_{k}n(x,\omega)=e^{-\pi i \langle x,\omega\rangle}e^{-\frac{\pi}{2}|\zeta|^2}\prod_{j=1}^d\overline{H_{n_j,k_j}(\zeta_j)},
    \end{align}
    where also \(\zeta\) is as in Lemma \ref{Wigner}.
\end{prop}
\begin{proof}
    Invoking Lemma \ref{Wigner}, the identity~\eqref{WignerIsSTFT} and the parity relation of Hagedorn wavepackets~\eqref{eq:parity} gives the result.
\end{proof}
As one would expect, the choice \(Q=\mathrm{Id}_d, P=i\mathrm{Id}_d\) recovers the Laguerre connection. Note also that up to the phase factor \(e^{-\pi i \langle x, \omega\rangle}\), \(V_kn(x,\omega)\) is just the STFT of the two Hermite functions $\phi_n$ and $\phi_k$, precomposed with the symplectic matrix \(T^{-1}\): \[|V_{k}n(x,\omega)|=\left|V_{\phi_k}\phi_n\left(T^{-1}\begin{pmatrix}
    x \\ \omega
\end{pmatrix}\right)\right|.\]
In light of this relation one should think of Hagedorn wavepackets as an explicit realization of the symplectic covariance property of the STFT~\cite{deGossonBook}.

\section{A generalized Daubechies' theorem for localization operators}\label{Sec:Daubechies}

In this section we present our main findings for classical localization operators. Our main tool will be the so-called \textit{double orthogonality relations}. It is well-known, see for example \cite{WhiteNoise,SeipRepFormula}, that a collection of functions \(\{\psi_n\}_{n\in\N_0}\) are the eigenfunctions of a localization operator \(A_F^g\) if and only if the short-time Fourier transforms \(\{V_g\psi_n\}_{n\in\N_0}\) are doubly orthogonal with respect to \(F\). That is, orthogonal both with respect to the measure \(dz\) and the measure \(F(z)\;dz\). We give the full statement of the characterization, but as we will see, this characterization is a special case of a more general characterization for mixed-state localization operators (Proposition~\ref{OperatorOrth}), and we therefore postpone the proof until the next section. Even so, we mention that the proof relies on using the spectral theorem on \(A_F^g\), which requires compactness. Compact localization operators have been characterized by Fernández and Galbis \cite[Prop. 3.6]{FG} in the following result.
\begin{lemma}[Fernández-Galbis]
    Let \(G\in \mathscr{S}_0'(\R^{2d})\) and \(g\in L^2(\R^d)\). The following are equivalent.
    \begin{itemize}
        \item \(A_G^g\) is compact on \(L^2(\R^d)\).
        \item There is a \(\Phi\in \mathscr{S}(\R^{2d})\setminus\{0\}\) such that for all \(R>0\), \[\lim_{|x|\to\infty}\sup_{|\omega|<R}|V_{\Phi}G(x,\omega)|=0.\]
    \end{itemize}
\end{lemma}
We call the distributions on \(\R^{2d}\) satisfying the second condition above the \textit{Fernández-Galbis class}, and denote it by \(FG(\R^{2d})\). Let us now state the double orthogonality criterion for localization operators. As mentioned, the proof is contained in Proposition~\ref{OperatorOrth}.
\begin{prop}\label{GeneralDblOrth}
    Let \(g\in L^2(\R^d)\), \(F\in FG(\R^{2d})\) and \(\{\psi_n\}_{n\in\N_0^d}\subset L^2(\R^d)\). Then \(\{\psi_n\}_{n\in\N_0^d}\) are the eigenfunctions of the localization operator \(A_{F}^g\) if and only if the STFTs \(\{V_g\psi_n\}_{n\in\N_0^d}\) form a complete subset of \(V_g(L^2)\) and are doubly orthogonal. That is, the identities
    \begin{align}\label{Dbl1}
        \int_{\R^{2d}} V_g\psi_n(z)\overline{V_g\psi_m}(z)\;dz=b_n\delta_{n,m}
    \end{align}
    and
    \begin{align}\label{Dbl2}
        \int_{\R^{2d}}F(z) V_g\psi_n(z)\overline{V_g\psi_m}(z)\;dz=c_n\delta_{n,m}
    \end{align}
    hold. In this case, the \(n\)-th eigenvalue is \(c_n\).
\end{prop}

Even though the Fernández-Galbis class contains all masks that define compact localization operators, there are several masks not contained in this class that have been studied using Daubechies' classical theorem. For instance, the complement of the disc centered at \(0\) with radius \(R\), \(D_R^c\), is not thin at infinity, but the Hermite functions are nevertheless the eigenfunctions of \(A_{D_R^c}^{\phi_0}\). We will therefore consider a simple extension of Fernández-Galbis' class where, while the operators may no longer be compact, we can still invoke the Courant-Fischer theorem, thereby obtaining eigenfunctions.
\begin{definition}
    We define the \textit{extended Fernández-Galbis class} by \begin{align*}
        EFG(\R^{2d})=\{F=c+G\colon c\in\C,G\in FG(\R^{2d})\}.
    \end{align*}
\end{definition}
By definition, a mask \(F\in EFG\) will give \[A_{F}^g=c\cdot A_1^g+A_G^g=c\cdot\mathrm{Id}_{L^2(\R^d)}+A_G^g.\]
That is, masks in EFG give scalings of the identity operator, perturbed by a compact localization operator. Compact perturbations of the identity will have eigenfunctions if and only if the compact operator does, and since by assumption \(G\in FG(\R^{2d})\), we are guaranteed that masks in \(EFG(\R^{2d})\) yield localization operators with eigenfunctions. This can also be easily seen by considering the localization problem, as we have \begin{align*}
    \int_{\R^{2d}} F(z) |V_gf(z)|^2\;dz &=\int_{\R^{2d}}( c+G(z)) |V_gf(z)|^2\;dz\\ &=c\int_{\R^{2d}} |V_gf(z)|^2\;dz+\int_{\R^{2d}} G(z) |V_gf(z)|^2\;dz,
\end{align*}
which, by Moyal's identity, equals \begin{align*}
    c\|f\|^2\|g\|^2+\int_{\R^{2d}} G(z) |V_gf(z)|^2\;dz=c\|g\|^2+\int_{\R^{2d}} G(z) |V_gf(z)|^2\;dz.
\end{align*}
So the functions that solve the localization problem for $F\in EFG(\R^{2d})$ are the same as the ones that solve the localization problem for \(G\in FG(\R^{2d})\). The \(n\)-th eigenvalues for \(A_F^g\) is the \(n\)-th eigenvalue of \(A_G^g\), plus \(c\|g\|^2\).\newline

There is one final property we need to impose on \(F\), namely polyradiality. Recall that a tempered distribution \(\varphi\) on \(\R^{2d}\) is said to be polyradial if for any Schwartz function \(f\in\mathscr{S}(\R^{2d})\) and any component-wise rotation, that is, a linear transformation \(A:\R^{2d}\rightarrow\R^{2d}\) on the form \(A(z_1,z_2,\dots, z_d)=(e^{i\theta_1}z_1,e^{i\theta_2}z_2,\dots, e^{i\theta_d}z_d)\), we have \begin{align*}
    \langle \varphi, f\rangle=\langle \varphi, f\circ A\rangle.
\end{align*}
Following \cite{Grafakos}, one can show that all polyradial distributions on \(\R^{2d}\) admit a unique representation on \(V^d\). With the necessary tools in hand, let us now state the first extension of Daubechies' theorem. The heuristic idea is this: If \(F\) is a polyradial tempered distribution, then if \(g\) is a Hermite function, so are all of the eigenfunctions, and the eigenvalues can be found by a weighted integral in absolute space.

 \begin{prop}[Generalized Daubechies' theorem]\label{Prop:WeightDaub}
     Let \(F\in EFG(\R^{2d})\) be polyradial, let \(F_0\) be its representation in \(V^d\) and let $k\in\N_0^d$. Then the eigenfunctions of the localization operator \(A_F^{\phi_k}\) are the product Hermite functions \(\{\phi_n\}_{n\in\N_0^d}\), and the \(n\)-th eigenvalue is \begin{align}\label{eq:eigFormula}
        c_{n,k}(F)=\frac{k!}{n!}\int_{V^d} u^{n-k}e^{-u} \left(\prod_{j=1}^dL_{k_j}^{n_j-k_j}(u_i)\right)^2 F_0\left(\sqrt{\frac{u}{\pi}}\right)\; du.
         \end{align}
 \end{prop}

\begin{proof}
    We must verify the double orthogonality conditions. Orthogonality with respect to Lebesgue measure follows from Moyal's identity. 
    To verify orthogonality with respect to \(F(z)\;dz\) note first that by using the Laguerre connection in each variable, as well as the decomposition from Equation (\ref{Decomp}), of \(V_{\phi_k}\phi_n\) we have that \begin{align}\label{STFTFormula}
        V_{\phi_k}\phi_n(z)=\prod_{j=1}^d \rho_{n_j,k_j}(r_j)e^{-i(n_j-k_j)\theta_j}e^{-i\pi\langle x_j,\omega_j\rangle}.
    \end{align} Inserting this into formula into the orthogonality condition~\eqref{Dbl2} we get \begin{align*}
        \int_{\R^{2d}}F(z)V_{\phi_k}\phi_n(z)\overline{V_{\phi_k}\phi_m(z)}\;dz &=\int_{\R^{2d}}F_0(r)\prod_{j=1}^d \rho_{n_j,k_j}(r_j)e^{-i(n_j-k_j)\theta_j}\rho_{m_j,k_j}(r_j)e^{i(m_j-k_j)\theta_j} \;dz \\
        &=\int_{\R^{2d}}F_0(r)\prod_{j=1}^d \rho_{n_j,k_j}(r_j)\rho_{m_j,k_j}(r_j)e^{i(m_j-n_j)\theta_j} \;dz.
    \end{align*}
    After a change to polycylindrical coordinates, this integral equals \begin{align*}
        &\int_{V^d}F_0(r)\prod_{j=1}^d \rho_{n_j,k_j}(r_j)\rho_{m_j,k_j}(r_j) r_j \;dr\cdot\int_{\T^d}\prod_{j=1}^d e^{i(m_j-n_j)\theta_j} \;d\theta \\
        &=\int_{V^d}F_0(r)\prod_{j=1}^d \rho_{n_j,k_j}(r_j)\rho_{m_j,k_j}(r_j) r_j \;dr\cdot (2\pi)^d \prod_{j=1}^d \delta_{n_j,m_j} \\
        &= (2\pi)^d \int_{V^d}F_0(r)\prod_{j=1}^d \left(\rho_{n_j,k_j}(r_j)\right)^2 r_j \;dr\cdot\delta_{n,m}.
    \end{align*}
    So we have verified~\eqref{Dbl2} with respect to \(F(z)\;dz\). To get the expression for the constant we insert the definition of \(\rho_{n_j,k_j}\) from Equation (\ref{Decomp}) and do the change of variables \(u=\pi r^2\):
 \begin{align*}
    c_{n,k}(F)&=(2\pi)^d\frac{k!}{n!}\pi^{n-k}\int_{V^d} r^{2n-2k+1}e^{-\pi |r|^2} \left(\prod_{j=1}^dL_{k_j}^{n_j-k_j}(\pi r_j^2)\right)^2 F_0(r)\; dr\\
    &=\frac{k!}{n!}\int_{V^d} u^{n-k}e^{-u} \left(\prod_{j=1}^dL_{k_j}^{n_j-k_j}(u_i)\right)^2 F_0\left(\sqrt{\frac{u}{\pi}}\right)\; du.
         \end{align*}
\end{proof}

If \(F\) is a distribution, the eigenvalue expression~\eqref{eq:eigFormula} is interpreted distributionally, and if \(n=0\) then Proposition~\ref{Prop:WeightDaub} recovers Daubechies' classical multivariate theorem. By covariance we can also extend this result to Hagedorn wavepackets.

\begin{corollary}[Daubechies' theorem for Hagedorn wavepackets]\label{HWPDblOrth}
    Let \(Q,P\) satisfy~\eqref{Lagrange}, and let $T$ be the corresponding symplectic matrix. Let \(F\in EFG(\R^{2d})\) be polyradial, and let \(F_0\) be its representation in \(V^d\). Let $k\in\N_0^d$, and consider the localization operator \(A_{F\circ T^{-1}}^{\phi_k[Q,P]}\). Its eigenfunctions are the Hagedorn wavepackets \(\{\phi_n[Q,P]\}_{n\in\N_0^d}\), and the \(n\)-th eigenvalue is the same as for \(A_F^{\phi_k}\), and therefore given by~\eqref{eq:eigFormula}.
\end{corollary}
\begin{proof}
    We must verify the double orthogonality conditions. Equation~\eqref{Dbl1} follows from Moyal's identity. For Equation~\eqref{Dbl2} we have from Proposition \ref{HagedornSTFT},  \begin{align*}
        \int_{\R^{2d}} F(T^{-1}z)V_{k}{n}(z)\overline{V_{k}{m}(z)}\;dz =\int_{\R^{2d}} F(T^{-1}z)  V_{\phi_k}\phi_n(T^{-1}(z))\overline{V_{\phi_k}\phi_m(T^{-1}(z))}\;dz. 
    \end{align*}
    Since \(T\) is symplectic, the linear change of variables \(z'=Tz\) has determinant \(1\). So, after this change of variables we get 
    \begin{align*}
        \int_{\R^{2d}}F(z)V_{\phi_k}\phi_n(z)\overline{V_{\phi_k}\phi_m(z)}\;dz,
    \end{align*}
    and the result follows.
\end{proof}

\subsection{Localizing on domains}
This section will focus on the special case when \(F=\chi_\Omega\) is a characteristic function. In terms of the spectrogram localization problem~\eqref{ConcProblem} this is the most interesting case. We will see that the conditions on \(\chi_\Omega\) from Proposition~\ref{Prop:WeightDaub} result in geometric conditions on \(\Omega\). Firstly, as seen by Fernández and Galbis, \(\chi_\Omega\in FG(\R^{2d})\) is equivalent to \(\Omega\) being \textit{thin at infinity}~\cite{FG2}. A set \(\Omega\) is thin at infinity if for one, and thus any, \(R>0\) we have
    \[\lim_{|z|\to\infty}\left|\Omega\cap B(z,R)\right|=0,\] where \(|\cdot|\) denotes the Lebesgue measure of the set. Bounded sets are thin at infinity, but sets with infinite measure can also be thin at infinity. One example is the set \[\Big\{z\in\R^2\colon \left|x^{4/3}-y^2\right|<1\Big\}.\]
On the other hand \(\chi_\Omega\in EFG(\R^{2d})\) if either \(\Omega\) or \(\Omega^c\) is thin at infinity. For our results we will only assume \(\Omega\) to be thin at infinity, as the extension to co-thin at infinity is clear.\newline

Secondly, let us treat polyradiality. The idea that polyradial functions are the natural multidimensional analogue of radially symmetric function is well-known, and goes back to Daubechies. However, it is less clear which domains produce polyradial characteristic functions. One of our contributions is noticing the connection between Reinhardt domains and polyradiality. In light of Lemma~\ref{Lem:Reinhardt} it is clear that \(\chi_\Omega\) being polyradial is equivalent to \(\Omega\) being a Reinhardt domain, as the decomposition~\eqref{eq:Polydecomp} also shows. Furthermore, the eigenvalues can be nicely interpreted as integrals over the Reinhardt shadow. We collect these findings in the result below. As usual, the same conclusion holds for Hagedorn wavepackets if instead \(\Omega\) is the symplectic image of a Reinhardt domain.

\begin{corollary}\label{MyDblOrth}
    Let $\Omega$ be Reinhardt domain with Reinhardt shadow $W$, and assume that \(\Omega\) is thin at infinity. Let $k\in\N_0^d$. Then, the eigenfunctions of the localization operator \(A_{\Omega}^{\phi_k}\) are the product Hermite functions \(\{\phi_n\}_{n\in\N_0^d}\). The \(n\)-th eigenvalue is \begin{align*}
        c_{n,k}(\Omega)=(2\pi)^d\frac{k!}{n!}\pi^{|n-k|}\int_W r^{2n-2k+1}e^{-\pi |r|^2} \left(\prod_{j=1}^dL_{k_j}^{n_j-k_j}(\pi r_j^2)\right)^2\; dr.
    \end{align*}
\end{corollary}

\begin{remark}
    By performing the change of variables \(u=\pi r^2\) the formula for \(c_{n,k}(\Omega)\) would completely align with Equation~\eqref{eq:eigFormula}, but we have opted to not do this, to keep the connection to the Reinhardt shadow clear.
\end{remark}

\begin{remark}
The proof of Corollary~\ref{MyDblOrth}, which is the exact same as the proof of Proposition~\ref{Prop:WeightDaub}, reveals that STFTs of Hermite functions are orthogonal on Reinhardt domains:\begin{align}\label{eq:ReinhardtDblOrth}
    \int_{\Omega}V_{\phi_k}\phi_n(z)\overline{V_{\phi_k}\phi_m(z)}\;dz=c_{n,k}(\Omega)\delta_{n,m}.
\end{align}
This equation should be thought of as the proper multidimensional analogue of Lemma~\ref{AbreuSpeckbacher}, which in \cite{LargeSieve}, Abreu and Speckbacher used to derive local reproducing formulas and sieving inequalities for the STFT on discs in \(\R^2\). The extension~\eqref{eq:ReinhardtDblOrth} suggests that similar results can be deduced in higher dimensions as well, and not just for discs, but for any Reinhardt domain. This is indeed the case, and for an account of this, as well as extensions to the noncommutative setting, see~\cite{QLS}.
\end{remark}
\subsection{Examples}

We will now present some special cases of the results from this section.

\subsubsection{Recovery of a known result}

In the case \(d=1\), Lemma \ref{MyDblOrth} is a statement about localization operators with circular mask and a Hermite function as window. In particular, if \(\Omega\) is a disc of radius \(R\) centered at \(0\), and \(\phi_k\) is the window function, then the eigenvectors of the localization operator \(A_{D_R}^{\phi_k}\) are the Hermite functions \(\{\phi_n\}_{n\in\N_0}\), and the eigenvalues are (resorting to the reflection identity \ref{Reflection} when \(n<k\))\begin{align*}
    \lambda_n &=\int_{D_R} |V_{\phi_k}\phi_n(z)|^2\;dz =2\pi \frac{k!}{n!}\pi^{n-k}\int_{0}^R r^{2n-2k+1}\left(L_{k}^{n-k}(\pi r^2)\right)^2 e^{-\pi r^2}\;dr \\&=\frac{k!}{n!}\int_0^{\pi R^2} t^{n-k}\left(L_{k}^{n-k}(t)\right)^2 e^{-t}\;dt.
\end{align*}
This result was shown in \cite[Prop.4.1]{HermiteDaub}, and was to our knowledge the best extension of Daubechies' theorem until now.\newline

\subsubsection{Localization on balls}

Our result extends the one from \cite{HermiteDaub} to localization operators in any dimension. If \(k\in\N_0^d\), then the STFT's of the Hermite functions, \(\{V_{\phi_k}\phi_n\}_{n\in\N_0^d}\) are orthogonal on any Reinhardt domain, and are thus the eigenfunctions of the corresponding localization operator \(A_{\Omega}^{\phi_k}\). The phase space ball is the most interesting example. If we let \(\Omega\subset \R^{2d}\) be the ball of radius \(R\) centered at \(0\), then its Reinhardt shadow is the ``quarter"-disc \(\{r\in V^d\colon \sum_{j=1}^d r_i^2\leq R^2\}\). Plugging this into Proposition \ref{Prop:WeightDaub} we get the following formula for the eigenvalues: \begin{align*}
    \lambda_n=(2\pi)^d\frac{k!}{n!}\pi^{|n-k|}\int_{\sum_{j=1}^d r_i^2\leq R^2} r^{2n-2k+1}e^{-\pi |r|^2} \left(\prod_{j=1}^dL_{k_j}^{n_j-k_j}(\pi r_j^2)\right)^2\; dr,
\end{align*}
which is valid in any dimension. Similar formulas can be deduced for any Reinhardt domain where one can describe the shadow.\newline

\subsubsection{A non-example}\label{ssec:Square}

Returning to \(d=1\), let \(Q_{a}\) be the square \([-a,a]\times[-a, a]\). \(Q_a\) is not a Reinhardt domain, so let us show that the Hermite functions cannot be the full set of eigenfunctions of the localization operator \(A_{Q_a}^{\phi_0}\). If they were, then we should have orthogonality of \(\{V_{\phi_0}\phi_n\}_{n\in\N_0}=\big\{\sqrt{\frac{\pi^n}{n!}}\overline{z}^ne^{-\frac{\pi}{2} |z|^2}\big\}_{n\in\N_0}\) on \(Q_{a}\). However, a combinatorial argument reveals that \begin{align*}
    \int_{-a}^a\int_{-a}^a \overline{z}^m z^n e^{-\pi |z|^2}\;dz=\begin{cases}
        \text{non-zero},\quad|m-n|=4k,\;k\in\N_0,\cr
        0\quad\text{otherwise. }
    \end{cases}
\end{align*}
 Thus, not all Hermite functions can be eigenfunctions of \(A_{Q_a}^{\phi_0}\). In fact, in~\cite{AbreuDorfler} it was even shown that \textit{no} Hermite function can be an eigenfunction of \(A_{Q_a}^{\phi_0}\)! This example suggests that if one wants the Hermite functions to be eigenfunctions, the domain needs full rotational symmetry. Partial rotational symmetry, like the square has, is not sufficient. For \(d=1\) this suspicion is again confirmed by the results of ~\cite{AbreuDorfler}, where it is shown that if \(\Omega\) is not fully rotationally invariant, then no Hermite functions can be eigenfunctions of \(A_{\Omega}^{\phi_0}.\)

\subsubsection{Projective localization}

Some choices of the mask \(F\in EFG(\R^{2d})\) admit a geometrical interpretation of \(A_F^g\), just like the operators \(A_{\Omega}^g\). Namely, we may interpret \(F\) as the curvature of a surface in \(\R^{2d+1}\), parametrized by phase space. For instance, we recognize the mask \[FS(x,w)=\frac{4}{(1+x^2+w^2)^2}=\left(\frac{2}{1+|z|^2}\right)^2\]
as the Fubini-Study metric of the Riemann sphere \(\C P^1\). We may therefore think of the localization operator \(A_{FS}^g\) as an operator that ``wraps" the STFT around the Riemann sphere, before synthesiszing back. This interpretation is justified by looking at the corresponding localization problem. We have \cite{ZelditchToric}\begin{align*}
    \langle A_{FS}^g f,f\rangle=\int_{\R^2}|V_g f(z)|^2 FS(z)\;dz=\int_{\C P^1} |V_gf(\zeta)|^2\;d\zeta=\|V_gf\|^2_{\C P^1}. 
\end{align*}
Thus, the eigenfunctions of \(A_{FS}^{g}\) are the functions whose STFT has the largest norm when considered as functions on the Riemann sphere. Since \(FS\) is a polyradial function (and it is in \(EFG(\R^2)\)), then if we let \(g=\phi_k\), Proposition \ref{Prop:WeightDaub} tells us that the eigenfunctions are the Hermite functions, and that the eigenvalues are 
\begin{align*}
    \lambda_n &=2\pi \frac{k!}{n!}\pi^{n-k}\int_{0}^\infty \frac{4}{(1+r^2)^2} r^{2n-2k+1}\left(L_{k}^{n-k}(\pi r^2)\right)^2 e^{-\pi r^2}\;dr
    \\&=4\pi^2\frac{k!}{n!}\int_0^{\infty}\frac{t^{n-k}}{(\pi+t)^2}\left(L_k^{n-k}(t)\right)^2e^{-t}\; dt.
\end{align*}

So the Hermite functions solve the localization problem on \(\C P^1\) when \(g=\phi_k\). The same is true for \(\C P^d\), as one can see by using the \(d\)-dimensional Fubini-Study metric as the mask of a localization operator. In general, the Hermite functions should solve any surface localization problem as long as the normal vector of the surface is polyradial and in \(EFG(\R^{2d})\).\newline

\subsubsection{Zero-diagonal wavepackets as solutions}
Let us return to the example in Section \ref{HagedornEx}. Pick the Lagrangian frame \(Q=\begin{pmatrix}
    1 & \frac{1-i}{\sqrt{2}} \\ \frac{1-i}{\sqrt{2}} & 1
\end{pmatrix}\) and \(P=\begin{pmatrix}
        i-1 & 0 \\
        0 & i-1
    \end{pmatrix}.\) The corresponding symplectic matrix \begin{align*}
    T=\begin{pmatrix}
        1 & \frac{1}{\sqrt{2}} & 0 & -\frac{1}{\sqrt{2}} \\
        \frac{1}{\sqrt{2}} & 1 & -\frac{1}{\sqrt{2}}& 0 \\
        -1 & 0 & 1 & 0 \\
        0 & -1 & 0 & 1 
    \end{pmatrix}
\end{align*}
describes how much one should deform the problem by if one wants to invoke Corollary~\ref{HWPDblOrth}. With this in mind, given any polyradial mask \(F \in EFG(\R^4)\), the Hagedorn wavepackets given by Equation \eqref{0Dformula} solve the eigenvalue problem for the localization operator \(A_{F\circ T^{-1}}^{\phi_k[Q,P]}\).

\section{Mixed-state localization operators}\label{Sec:Mixed-State}

We now want to solve the eigenvalue problem for mixed-state localization operators. Recall that the mixed-state localization operator with mask \(F\) and window \(S\) is defined as the function-operator convolution \(F\star S\). We will use quantum harmonic analysis to generalize results from the previous sections to the mixed-state case. The first step is the following quantum analogue of double orthogonality.

\begin{prop}\label{OperatorOrth}
    Let \(S\) be an operator on \(L^2(\R^d)\), and \(F\) a tempered distribution on \(\R^{2d}\) such that the mixed-state localization operator \(F\star S\) is compact. Let \(\{\psi_n\}_{n\in\N_0^d}\) be a collection of functions in \(L^2(\R^d)\). Then \(\{\psi_n\}_{n\in\N_0^d}\) are the eigenfunctions of \(F\star S\) if and only if the following conditions are satisfied
    \begin{align}\label{QDbl1}
        \{\psi_n\}_{n\in\N_0^d}\text{ is an orthonormal basis of }L^2(\R^d),
    \end{align}
    and
    \begin{align}\label{QDbl2}
        \int_{\R^{2d}} F(z)Q_{S}\left(\psi_n,\psi_m\right)(z)\;dz=c_n\delta_{n,m}
    \end{align}
    hold. In this case, the \(n\)-th eigenvalue is \(c_n\). We call the two conditions the \textbf{quantum double orthogonality} conditions.
\end{prop}
\begin{proof}
    First assume that the eigenfunctions of \(F\star S\) are \(\{\psi_n\}_{n\in\N_0^d}\). By the spectral theorem the eigenfunctions of the operator \(F\star S\) form an ONB of \(L^2(\R^d)\). Using the orthogonality of \(\{\psi_n\}_{n\in\N_0^d}\), as well as the fact that they are eigenfunctions, we get \begin{align*}
        \lambda_n \delta_{n,m} &=\lambda_n\langle \psi_n,\psi_m \rangle=\langle \lambda_n \psi_n,\psi_m\rangle=\langle F\star S(\psi_n),\psi_m\rangle \\
        &=\int_{\R^{2d}}F(z)\langle \alpha_z(S)\psi_n,\psi_m\rangle\;dz=\int_{\R^{2d}}F(z)\mathrm{tr}\left( (\psi_n\otimes \psi_m)\alpha_z(S)\right)\;dz
        \\ &=\int_{\R^{2d}}F(z)\left(\left(\psi_n\otimes\psi_m\right)\star\check{S}\right)(z)\;dz=\int_{\R^{2d}}F(z)Q_{S}(\psi_n,\psi_m)(z)\;dz.
    \end{align*}
    Let us now prove the converse. Assume that \(\{\psi_n\}_{n\in\N_0^d}\) is an ONB satisfying the integral identity above. Since \(\{\psi_n\}_{n\in\N_0^d}\) is an ONB, we may expand \(F\star S\) in terms of its matrix coefficients. That is, there are coefficients \(a_{m}\) such that \begin{align}\label{LinComb}
        F\star S(\psi_n)=\sum_{m\in\N_0^d} a_{m}\psi_m.
    \end{align}
    By definition, the coefficients are \begin{align*}
        a_{m}=\langle F\star S(\psi_n),\psi_m\rangle=\int_{\R^{2d}} F(z)Q_{S}\left(\psi_n,\psi_m\right)(z)\;dz=c_n\delta_{n,m}.
    \end{align*}
    Plugging this into the expansion, we have \begin{align*}
        F\star S(\psi_n)=c_n\psi_n\quad \forall \;n\in\N_0^d.
    \end{align*}
    So \(\{\psi_n\}_{n\in\N_0^d}\) is a complete set of eigenfunctions for \(F\star S\).
\end{proof}
\begin{remark}\label{Rem:DblQdbl}
    When \(S\in\mathcal{S}^1\), the condition~\eqref{QDbl1} is equivalent to \(\int_{\R^{2d}}Q_{S}\left(\psi_n,\psi_m\right)(z)\;dz=b_n\delta_{n,m}\), a consequence of the generalized Moyal identity. However, since we also consider operators not in the trace class, we have stated~\eqref{QDbl1} in a way which does not require the trace of \(S\) to be finite.
\end{remark}
The assumptions of Proposition~\ref{OperatorOrth} are less concrete than those of Proposition~\ref{GeneralDblOrth}. In particular, verifying the compactness of \(F\star S\) is nontrivial, so let us therefore briefly discuss some sufficient conditions for compactness. If \(S\in \mathcal{S}^1\) then the continuity of function-operator convolution \cite{SchwartzOps} assures that compactness is equivalent to \(F\in FG(\R^{2d})\). Moreover, since every \(F\in EFG(\R^{2d})\) can be written as \(c+G\) with \(G\in FG(\R^{2d})\), we have that \(F\star S=c\star S+G\star S=c\cdot\mathrm{tr}(S)\cdot \mathrm{Id}_{L^2(\R^d)}+F\star S\). Arguing as in the previous section we can also use the quantum double orthogonality characterization on \(F\star S\) with \(F\in EFG(\R^{2d})\), too.

For operators not in \(\mathcal{S}^1\) we instead resort to the Tauberian theorem for operators \cite{bible3}. In particular, Tauberian theory gives the following characterization of compactness when \(F\in L^1(\R^{2d})\)~\cite[Prop.6.1]{bible3}: Given \(S\in\mathcal{B}(L^2(\R^d))\) and \(F\in L^1(\R^{2d})\), \(F\star S\) is compact if and only if there exists some \(\phi\in L^2(\R^d)\) with \(V_{\phi}\phi\) zero-free such that the Cohen's class \(Q_S(\phi)\) is a continuous function vanishing at infinity. Note that this criterion is independent of \(F\), making it a very useful tool. For example, it easily shows that the Weyl transform \(L_F\) is compact for  all \(F\in L^1(\R^{2d})\). We recognize \(L_F\) as the mixed-state localization operator \(F\star(2^d P)\), and if we let \(\phi=\phi_0\), then the corresponding Cohen's class is\begin{align*}
        Q_{2^dP}(\phi_0)(z)=W(\phi_0)(z)=2^de^{-2\pi |z|^2}, 
    \end{align*}
which is clearly continuous and vanishing at infinity. Similarly, \(F\star S\) will also be compact for any \(F\in L^1(\R^{2d})\) and any \(S\) with Weyl symbol in \(L^2(\R^{2d})\cap L^1(\R^{2d})\), which again can be seen by considering the Cohen's class distribution of the Gaussian. The two methods discussed above are not the only ways of ensuring compactness of \(F\star S\), but for our purposes they will be enough.\newline

To illustrate quantum double orthogonality, let us first consider the rank-one case \(S=g\otimes g\). Then \(Q_{S}\left(\psi_n,\psi_m\right)(z)=V_g\psi_n(z)\overline{V_g\psi_m(z)}\), and so~\eqref{QDbl2} reduces to the classical double orthogonality relation~\eqref{Dbl2}, and by Remark~\ref{Rem:DblQdbl},~\eqref{QDbl1} is in this case equivalent to~\eqref{Dbl1}. So Proposition \ref{OperatorOrth} really is a generalization of classical double orthogonality.\newline

If \(S=2^d P\) and \(F\in L^1(\R^{2d})\), then \(Q_{S}\left(\psi_n,\psi_m\right)(z)=W(\psi_n,\psi_m)(z)\), and the orthogonality condition \eqref{QDbl2} becomes \begin{align*}
    \int_{\R^{2d}} F(z)W(\psi_n,\psi_m)(z)\;dz=c_n\delta_{n,m}. 
\end{align*}
In \cite{LiebOstrover}, Lieb and Ostrover showed that the Hermite functions \(\{\phi_n\}_{n\in\N_0^d}\) satisfy this relation when \(F(z)\) is the characteristic function of a ball with radius \(R\). It was also implicitly used in the one-dimensional case by Flandrin and Ramanathan-Topiwala \cite{flandrin,RamTop2}. Proposition~\ref{OperatorOrth} thus captures both the classical double orthogonality of STFTs and techniques used for the Wigner concentration problem studied by the above mentioned authors. Unlike the classical case this identity can not be reduced to a classical orthogonality statement, as \(\psi_n\) and \(\psi_m\) are mixed together inside the integral.\newline

If \(S\) is a positive trace class operator, then we can use the square root to factorize the Cohen's class: \begin{align*}
    Q_{S}\left(\psi_n,\psi_m\right)(z)&=\langle \alpha_z(S)\psi_n,\psi_m\rangle=\langle \pi(z)\sqrt{S}\sqrt{S}\pi(z)^*\psi_n,\psi_m\rangle\\&=
    \langle \sqrt{S}\pi(z)^*\psi_n,\sqrt{S}\pi(z)^*\psi_m\rangle=\langle \mathfrak{V}_{\sqrt{S}}\psi_n\otimes e,\mathfrak{V}_{\sqrt{S}}\psi_m\otimes e\rangle_{\mathcal{S}^2},
\end{align*} 
where \(e\) is any normalized function in \(L^2(\R^d)\). So for every positive trace class operator Proposition \ref{OperatorOrth} reduces to a statement about double orthogonality of the operator STFTs \(\{\mathfrak{V}_{\sqrt{S}}\psi_n\otimes e\}_{n\in\N_0^d}\) in the Bochner space \(L^2(\R^{2d},\mathcal{S}^2)\). 
We will make use of this connection in Section~\ref{Sec:Toeplitz}.\newline

With the examples in mind, let us now formulate quantum versions of Daubechies' theorem. Our main tool will be Proposition \ref{OperatorOrth}, but even without it, one can use continuity and linearity to formally extend the generalizations of Daubechies' theorem from the last section to operators of the form \begin{align*}
    S=\sum_{|k|=0}^\infty \alpha_k\phi_k[Q,P]\otimes\phi_k[Q,P].
\end{align*}
This extenstion works well in theory, but is not useable in practice since it requires an explicit decomposition of \(S\) in terms of Hagedorn states. It also does not provide a closed form expression for the eigenvalues. We mitigate this problem by instead imposing a condition on the Fourier-Wigner transform of the operator. The class of operators we will now consider are the following. 
\begin{definition}
    Let \(S \) be an operator on \(L^2(\R^d)\) with kernel in \(\mathscr{S}'(\R^{2d})\). If the Fourier-Wigner transform \(\mathcal{F}_W(S)\) is a radial tempered distribution, we call \(S\) a \textit{radial operator}. If \(\mathcal{F}_W(S)\) is polyradial, we call \(S\) a \textit{polyradial operator}.
\end{definition}
Since the Weyl symbol of \(S\), \(a_S\), is \(\mathcal{F}_{\sigma}\mathcal{F}_W(S)(z)\), and the symplectic Fourier transform preserves polyradiality, a polyradial operator can equivalently be defined as an operator with a polyradial Weyl symbol. Several familiar operators are polyradial, and we collect some of them in the following example.
\begin{exmp}
    \begin{itemize}
        \item The Fourier-Wigner transform of a pure state \(g\otimes g\) is known to be the ambiguity function, \(e^{\pi i \langle x, \omega\rangle} V_gg.\) It is well known~\cite[Corollary 4.77]{Folland} that the only polyradial pure states are the Hermite pure states \(\phi_n\otimes\phi_n\). If \(d>1\) the pure state \(\phi_0\otimes\phi_0\) is the only radial pure state.
        \item The parity operator, \(2^dPf(t)=2^df(-t)\), has Fourier-Wigner transform \(1\), and is thus a radial operator. 
        \item The Hamiltonian of the quantum harmonic oscillator, \(Hf(t)=\frac{1}{2}\left( t^2f(t)-\frac{\Delta f(t)}{4\pi}\right)\), is the Weyl transform of \(\frac{x^2+\omega^2}{2}\), 
        which is a radial function. Thus \(H\) is a radial (unbounded) operator.
        \item From \cite[Section 2.1]{Folland}, the Fourier transform has Weyl symbol \((1+i)^de^{-2\pi i|z|^2}\), so it is a radial operator.
        \item By Pool's theorem \cite{Pool}, any polyradial function \(f\in L^2(\R^d)\) defines a polyradial Hilbert-Schmidt operator via the Weyl transform.
    \end{itemize}
\end{exmp}

Out of these examples, only the pure Hermite states \(\phi_n\otimes \phi_n\) are polyradial trace class operators. Note however that there are polyradial trace class operators that are not pure. An immediate example is of course given by finite linear combinations \(\sum_{|n|=0}^N \lambda_n\phi_n\otimes\phi_n\), but infinite rank examples also exist. For instance, in \cite{Thermal} the Gaussians that produce positive trace class operators were characterized. If \(J\) denotes the \(2d\times 2d\)-matrix \(\begin{pmatrix}
    0 & \mathrm{Id}_d \\ -\mathrm{Id}_d & 0
\end{pmatrix}\) and \(M\) is a positive, symmetric \(2n\times 2n\)-matrix, then the Weyl transform of the Gaussian \begin{align}\label{GeneralGaussian}
    g_M(z)=(2\pi)^{-d}\sqrt{\mathrm{det}(M^{-1})}e^{-\frac{1}{2}\langle M^{-1}z,z\rangle}
\end{align} is a positive trace class operator if and only if \(M+\frac{i}{4\pi}J\) is positive semidefinite \cite[Prop. 21]{Thermal}. Clearly, these operators are only polyradial if \(M\) is diagonal and the diagonal entries corresponding to \(x_i\) and \(\omega_i\) agree. In one dimension, these operators are called \textit{thermal states,} and they are known to be mixed states. In particular, given an energy \(E>0\), the matrix \begin{align*}
    E=\begin{pmatrix}
        \frac{1}{4\pi}+\frac{E}{2\pi} & 0 \\ 0 & \frac{1}{4\pi}+\frac{E}{2\pi}
    \end{pmatrix}
\end{align*}
gives a Gaussian whose Weyl transform factors as \begin{align}\label{ThermalState}
    L_{g_E}=\frac{1}{E+1}\sum_{n=0}^\infty \frac{E^n}{(E+1)^n}\phi_n\otimes\phi_n,
\end{align}
and is thus not a pure state.
\newline

We are now ready for the two main results of this section. Given a polyradial mask \(F\) and a polyradial operator \(S\), the eigenfunctions of the mixed-state localization operator \(F\star S\) are given by a Daubechies-type theorem.

\begin{prop}[Daubechies' theorem for polyradial operators]\label{OperatorDaub}
    Let \(S \) be a polyradial operator on \(L^2(\R^d)\), and \(F\) be a polyradial mask such that the mixed-state localization operator \(F\star S\) is compact. Then the eigenfunctions of \(F\star S\) are the product Hermite functions \(\{\phi_n\}_{n\in\N_0^d}\). The \(n\)-th eigenvalue is  \begin{align}\label{operatorFormula}
    \begin{split}
        \lambda_n &=\int_{\R^{2d}} (F*a_S)(z)\cdot W(\phi_n)(z)\;dz\\
        &=(-1)^{|n|}\int_{V^d}F*a_S\left(\sqrt{\frac{u}{2\pi}}\right)\left(\prod_{j=1}^d L_{n_j}^0(2u)\right)e^{-u}\;du.
    \end{split}\end{align}
\end{prop}
\begin{proof}
    With Proposition \ref{OperatorOrth} in mind, all we need to prove is the orthogonality condition~\eqref{QDbl2}. Since \(S\) is polyradial, we have \(\mathcal{F}_W(\check{S})(z)=\mathcal{F}_W(S)(z)\), which follows from the definition, as well as the identity \(P\pi (z) P=\pi (-z)\). Thus we also have \(a_S=a_{\check{S}}\). The integral now becomes \begin{align*}
        \int_{\R^{2d}} F(z)Q_{S}\left(\phi_n,\phi_m\right)(z)\;dz &=\int_{\R^{2d}} F(z)\left(a_{\check{S}}*W(\phi_n,\phi_m)\right)(z)\;dz\\
        &=\int_{\R^{2d}} F(z)\left(a_{S}*W(\phi_n,\phi_m)\right)(z)\;dz \\ &=\int_{\R^{2d}} \left(F*Pa_{S}\right)(z)\cdot W(\phi_n,\phi_m)(z)\;dz \\ 
        &=\int_{\R^{2d}} \left(F*a_{S}\right)(z)\cdot W(\phi_n,\phi_m)(z)\;dz
    \end{align*}
    where we used Fubini's theorem in the last line. Since convolution preserves polyradiality, the convolution in the last line will be polyradial, and for simplicity, we will write  \begin{align*}
        G(z)=\left(F*a_S\right)(z).
    \end{align*} Thus, we have 
    \begin{align*}
        \int_{\R^{2d}}G(z)\cdot W(\phi_n,\phi_m)(z)\;dz &=\int_{\R^{2d}} G(z) 2^d e^{4\pi i \langle x,\omega\rangle} V_{\check{\phi}_m}\phi_n(2z)\;dz\\ &=2^d(-1)^{|m|}\int_{\R^{2d}} G(z)  e^{4\pi i \langle x,\omega\rangle} V_{\phi_m}\phi_n(2z)\;dz
        \\&=2^d(-1)^{|m|}\int_{\R^{2d}} G(z)  e^{-2\pi|z|^2}H_{n,m}(2z)\;dz\\ &=2^d(-1)^{|m|}\int_{\R^{2d}} G(z) \rho_{n,m}(2r)e^{i(m-n)\theta}\;dz,
    \end{align*}
     where \(\rho_{n,m}\) is as in Equation \eqref{Decomp}. Orthogonality now follows by a change to polycylindrical coordinates. The formula for the eigenvalue follows by setting \(n=m\):
     \begin{align*}
         \lambda_n &=2^d(-1)^{|n|}\int_{\R^{2d}} G(z) \rho_{n,n}(2r)\;dz\\&=(4\pi)^d(-1)^{|n|}\int_{V^d}\left(F*a_S\right)(r)\left(\prod_{j=1}^d L_{n_j}^0(4\pi r_j^2)\right)e^{-2\pi r^2}r\;dr \\
        &=(-1)^{|n|}\int_{V^d}\left(F*a_S\right)\left(\sqrt{\frac{u}{2\pi}}\right)\left(\prod_{j=1}^d L_{n_j}^0(2u)\right)e^{-u}\;du.
     \end{align*}
\end{proof}
The eigenvalue formula~\eqref{operatorFormula} generalizes a one-dimensional formula of Ramanathan-Topiwala \cite{RamTop2} and is the most generally useful formula for the eigenvalues. Note however that it can sometimes be more instructive to use any of the following equivalent formulas:\begin{align}\label{alternate}
        \lambda_n &=\int_{\R^{2d}} (F*a_S)(z)\cdot W(\phi_n)(z)\;dz\\&=\int_{\R^{2d}} F(z)\left(a_{S}*W(\phi_n)\right)(z)\;dz=
        \int_{\R^{2d}} (F*W(\phi_n))(z)\cdot a_{S}(z)\;dz.
    \end{align}
What formula to use will depend on which convolution is the easiest to compute. For instance, we will make use of the second formula in Section~\ref{ssec:Gaussian}.

Lastly, let us state the Hagedorn wavepacket version of Proposition \ref{OperatorDaub}, which we will need later. Since the result follows directly from Proposition~\ref{OperatorDaub} and covariance, we state it as a corollary. 
\begin{corollary}\label{OperatorHagedorn}
    Let \(R \) be a polyradial Schwartz operator on \(L^2(\R^d)\), and let \(T\) be a \(d\times d\) symplectic matrix with block form \(\begin{pmatrix}
        A & B \\ C& D
    \end{pmatrix}.\) Let \(S\) be the operator with Weyl symbol \(a_S=a_R\circ T^{-1}.\) Let  \(F\) be a polyradial mask such that \((F\circ T^{-1})\star S\) is compact. Then, the eigenfunctions of \((F\circ T^{-1})\star S\) are the Hagedorn wavepackets \(\{\phi_n[Q,P]\}_{n\in\N_0^d}\), where \(Q=A+iB\) and \(P=C+iD\). The \(n\)-th eigenvalue is \begin{align*}
        \lambda_n=\int_{\R^{2d}} F(z)Q_{R}(\phi_n)(z)\;dz.
    \end{align*}
\end{corollary}

We will now consider two applications of our results.
\subsection{A variant of Lerner's problem}
In~\cite{Lerner}, Lerner asked the following question about the Weyl transform of \(p\)-norm balls in \(\R^{2d}\): \begin{conjecture}[{\cite[Question 8.3]{Lerner}}]
    Let \(p\in [1,\infty]\setminus\{2\}\) and let \(\mathbb{B}_{p}^{2d}(R)\) denote the \(p\)-norm ball of radius \(R\) in \(\R^{2d}\) (with the obvious adjustment when \(p=\infty\)):\[\mathbb{B}_{p}^{2d}(R)=\big\{(x,\omega)\in \R^{2d}\colon \sum_{i=1}^{d}x_i^p+\omega_i^p\leq R^p\big\}.\] Is it possible to say anything about the spectrum of \(L_{\chi_{\mathbb{B}_{p}^{2d}(R)}}\)?
\end{conjecture} 
Lerner's problem is related to the Wigner localization problem~\eqref{WignerProblem}, and even for \(d=1\), it is highly nontrivial. Similarly to Example~\ref{ssec:Square} the challenges can likely be attributed to the limited rotation invariance of the \(p\)-balls. While we do not claim that our results solve Lerner's problem for \(p\)-balls in \(\R^{2d}\), the situation is completely different if we consider \(p\)-balls in \(\C^d\) instead. Recall from Example~\ref{Ex:Reinhardts} that for any \(p\in (0,\infty]\), the \(p\)-balls \begin{align}\label{eq:newBall}
     \mathfrak{B}_{p}^{d}(R)=\big\{(x,\omega)\in \R^{2d}\colon \sum_{i=1}^{d}|x_i+i\omega_i|^{p}\leq R^p\big\}
\end{align} are Reinhardt domains with shadow \[
W_p(R)=\big\{r\in V^d\colon \sum_{i=1}^d r_i^p\leq R^p \big\}.
\]
Furthermore, since the Weyl transform of \(F\) may be written as the mixed-state localization operator \(F\star 2^dP\), Daubechies' theorem for mixed-state localization operators applies, and we get a solution to our Lerner-like problem.
\begin{corollary}[Solution to Lerner's problem in $\C^d$]
    Let \(p\in (0,\infty]\) and \(0<R<\infty\), and let \( \mathfrak{B}_{p}^{d}(R)\) be as in \eqref{eq:newBall}. The eigenfunctions of the Weyl transform \(L_{\chi_{ \mathfrak{B}_{p}^{d}(R)}}=\chi_{ \mathfrak{B}_{p}^{d}(R)}\star 2^dP\) are the product Hermite functions, and the eigenvalue corresponding to \(\phi_n, n\in \N_0^d\) is \begin{align*}
        \lambda_n=(4\pi)^d(-1)^{|n|}\int_{W_p(R)}e^{-2\pi r^2} L_n^0(4\pi r^2)r\;dr.
    \end{align*}
\end{corollary}
\begin{proof}
    As \(R<\infty\), we have \(\chi_{ \mathfrak{B}_{p}^{d}(R)}\in L^1(\R^{2d})\) and thus \(\chi_{ \mathfrak{B}_{p}^{d}(R)}\star 2^dP\) is compact. Proposition~\ref{OperatorDaub} now applies. 
\end{proof}
For balls in \(\C^d\) we are even able to describe the spectrum for the non-convex case \(p\in (0,1)\), which further illustrates the advantages compared to balls in \(\R^{2d}\).\newline

\subsection{Consequences for Gaussian Cohen's classes}\label{ssec:Gaussian}
We have seen that letting \(S=\phi_n\otimes \phi_n\) or \(2^dP\) in Proposition \ref{OperatorDaub} recovers the known cases of the localization operators and the Weyl transform, respectively. 
Let us now illustrate the use of Proposition \ref{OperatorDaub} and Corollary \ref{OperatorHagedorn} by looking at a more exotic example: mixed-state localization operators on the form \(F\star L_{g_M}\), where \(g_M\) is a general Gaussian (see~\eqref{GeneralGaussian}). These Gaussians are significant in time-frequency analysis not just because their Weyl transforms are quantum states, but also because their Cohen's class distributions, \(Q_{g_M}(f)=W(f)*g_M\) are positive functions \cite[Theorem 4.4.4]{Grochenig}. From what we know the spectral properties of these operators have not been studied before. We want to use Proposition \ref{OperatorDaub} to find the eigenfunctions and eigenvalues of \(F\star L_{g_M}\), but as already noted, only the Gaussians where \(M=\mathrm{diag}(m_1,m_2,\dots, m_d,m_1,m_2,\dots , m_d)\) give polyradial operators. Still, we can solve the eigenvalue problem for \textit{any} covariance matrix \(M\) by using Williamson's diagonalization theorem \cite[Prop. 4.22]{Folland}. Since \(M\) is symmetric and positive definite, there is a symplectic \(2d\times 2d\)-matrix \(T\) and a diagonal matrix \(K=\mathrm{diag}(k_1,k_2,\dots, k_d,k_1,k_2,\dots , k_d)\) with \(k_i>0\) such that \begin{align}\label{Williamson}
    M=TKT^T.
\end{align}
Plugging this into the definition of \(g_M\), we see that \begin{align*}
    g_M(z) &=(2\pi)^{-d}\sqrt{\mathrm{det}(M^{-1})}e^{-\frac{1}{2}\langle M^{-1}z,z\rangle}\\
    &=(2\pi)^{-d}\sqrt{\mathrm{det}(T^{-T}K^{-1}T^{-1})}e^{-\frac{1}{2}\langle T^{-T}K^{-1}T^{-1}z,z\rangle}\\
    &=(2\pi)^{-d}\sqrt{\mathrm{det}(K^{-1})}e^{-\frac{1}{2}\langle K^{-1}T^{-1}z,T^{-1}z\rangle}=g_K\left(T^{-1}z\right).
\end{align*}
So for any admissible \(M\), the trace class operator \(L_{g_M}\) satisfies the assumptions of Corollary \ref{OperatorHagedorn}, and we get the eigenfunctions and eigenvalues for suitable masks \(F\). 
\begin{corollary}
    Let \(M\in\R^{2d\times 2d}\) be a positive definite, symmetric matrix such that \(M+\frac{J}{4\pi}\) is positive semidefinite. Let \(K\) and \(T\) denote the matrices in the Williamson diagonalization of \(M\), as in ~\eqref{Williamson}, and write \(T\) on block form: \(T=\begin{pmatrix}
        A & B \\ C & D
    \end{pmatrix}\). Let \(G\in EFG(\R^{2d})\) be polyradial, and let \(F(z)=G\left(T^{-1}z\right)\). Then the eigenfunctions of the mixed-state localization operator \(F\star L_{g_M}\) are the Hagedorn wavepackets \(\{\phi_n[A+iB,C+iD]\}_{n\in\N_0^d}\), and the eigenvalues are \begin{align}\label{GaussFormula}
    \begin{split}
        \lambda_n &=\int_{\R^{2d}} G(z)\left(g_{K}*W(\phi_n)\right)(z)\;dz\\
    &=(2\pi)^d\prod_{j=1}^d\left(\frac{4\pi k_j-1}{4\pi k_j+1}\right)^{n_j}\int_{V^d} G(r)r g_{ K+\frac{I}{4\pi}}(r)\left(\prod_{j=1}^d L_{n_j}^0\left(\frac{4\pi r_j^2}{1-16\pi^2k_j^2}\right)\right)\;dr.        
    \end{split}
    \end{align}
\end{corollary}
The result above is just a restatement of Corollary \ref{OperatorHagedorn} for the case \(S=L_{g_M}\). 
The final eigenvalue formula is specific to the Gaussians, however, and relies heavily on properties of the Gaussians and the Laguerre functions. The main challenge is computing the convolution between a Gaussian and a Wigner function. We do this computation in the case \(d=2\) below. The higher-dimensional case then follows from tensorization. In order to simplify the computation we will use a different convention for the Gaussian.
\begin{lemma}
        Let \(n\in\N_0, E>\frac{1}{2}\), and consider the \(2\)-dimensional heat kernel \(\gamma_E(z)=\frac{1}{E}e^{-\pi\frac{|z|^2}{E}}\) and the Hermite Wigner function \(W(\phi_n)\). Then \begin{align*}
        \left(\gamma_E*W(\phi_n)\right)(z)=\left(\frac{E-\frac{1}{2}}{E+\frac{1}{2}}\right)^n L_n^0\left(\frac{\pi |z|^2}{\frac{1}{4}-E^2}\right)\gamma_{E+\frac{1}{2}}(z).
    \end{align*}
\end{lemma}
\begin{proof}
    By Fourier transform, we have \begin{align*}
    \gamma_E*W(\phi_n) &=\mathcal{F}^{-1}\mathcal{F}\left(\gamma_E*W(\phi_n)\right)=\mathcal{F}^{-1}\left(e^{-\pi E |z|^2}\cdot A(\phi_n)\right)\\
    &=\mathcal{F}^{-1}\left(e^{-\pi E |z|^2}\cdot L_n^0(\pi |z|^2)e^{-\frac{\pi}{2}|z|^2}\right)=\mathcal{F}^{-1}\left( L_n^0(\pi |z|^2)e^{-\pi\left(E+\frac{1}{2}\right)|z|^2}\right).\end{align*}
    Let \(\Tilde{E}=E+\frac{1}{2}\). Recalling that the Fourier transform turns \(\pi|z|^2 f\) into \(-\frac{\Delta\hat{f}}{4\pi} \), the expression above now equals \begin{align*}
    \mathcal{F}^{-1}\left( L_n^0(\pi |z|^2)e^{-\pi\Tilde{E}|z|^2}\right)=\left(L_n^0\left(-\frac{\Delta}{4\pi}\right)\right)\gamma_{\Tilde{E}}(z). 
    \end{align*}
    Now note that \(\gamma_{\tilde{E}}\) solves the two-dimensional heat equation with diffusion constant \(\frac{1}{4\pi}\): \begin{align*}
        \frac{\partial \gamma_{\tilde{E}}}{\partial \tilde{E}}=\frac{1}{4\pi}\Delta \gamma_{\tilde{E}}.
    \end{align*}
    Furthermore, a proof by induction shows that \begin{align*}
        \frac{\partial^n \gamma_{\tilde{E}}}{\partial \tilde{E}^n}=(-1)^n\frac{n!}{\tilde{E}^n}L_n^0\left(\frac{\pi |z|^2}{\tilde{E}}\right)\gamma_{\tilde{E}}.
    \end{align*} Returning to the expression above with these facts in mind, we get 
    \begin{align*} \left(L_n^0\left(-\frac{\Delta}{4\pi}\right)\right)\gamma_{\Tilde{E}}(z) &=\left(L_n^0\left(-\frac{\partial}{\partial \Tilde{E}}\right)\right)\gamma_{\Tilde{E}}(z) 
   =\sum_{j=0}^n(-1)^j\binom{n}{n-j}\frac{1}{j!}\left(-\frac{\partial}{\partial \Tilde{E}}\right)^j\gamma_{\Tilde{E}}(z)\\&=\sum_{j=0}^n\binom{n}{n-j}\frac{1}{j!}\left(\frac{\partial}{\partial \Tilde{E}}\right)^j\gamma_{\Tilde{E}}(z)
    \\ &=\sum_{j=0}^n\binom{n}{n-j}\frac{1}{j!}(-1)^j\frac{j!}{\Tilde{E}^j}L_j^0\left(\frac{\pi|z|^2}{\Tilde{E}}\right)\gamma_{\Tilde{E}}(z)\\
    &=\gamma_{\Tilde{E}}(z)\sum_{j=0}^n\binom{n}{n-j}\left(-\frac{1}{\Tilde{E}}\right)^jL_j^0\left(\frac{\pi|z|^2}{\Tilde{E}}\right).
\end{align*}
We now wish to use the scaling formula for the Laguerre polynomials~\cite{Dunkl}: For any \(t,b\in \R\), we have \begin{align*}
    L_n^0(bt)=\sum_{j=0}^n\binom{n}{j}b^j(1-b)^{n-j}L_j^0(t).
\end{align*}
We want \(b^j(1-b)^{-j}=\left(-\frac{1}{\Tilde{E}}\right)^j\), and a simple calculation shows that \(b=\frac{1}{1-\Tilde{E}}\) does the trick, and that we have \(1-b=\frac{\Tilde{E}}{\Tilde{E}-1}\). However, we must also divide by \(\left(\frac{\Tilde{E}}{\Tilde{E}-1}\right)^n\) to account for the missing term inside the sum. Thus, our expression above equals\begin{align*}
    &\left(\frac{\Tilde{E}-1}{\Tilde{E}}\right)^n \gamma_{\Tilde{E}}(z)\sum_{j=0}^n\binom{n}{n-j}\left(\frac{1}{1-\Tilde{E}}\right)^j\left(\frac{\Tilde{E}}{\Tilde{E}-1}\right)^{n-j}L_j^0\left(\frac{\pi|z|^2}{\Tilde{E}}\right)\\
    &=\left(\frac{\Tilde{E}-1}{\Tilde{E}}\right)^n \left(\sum_{j=0}^n\binom{n}{j}\left(\frac{1}{1-\Tilde{E}}\right)^j\left(\frac{\Tilde{E}}{\Tilde{E}-1}\right)^{n-j}L_j^0\left(\frac{\pi|z|^2}{\Tilde{E}}\right)\right)\gamma_{\Tilde{E}}(z)\\
    &=    \left(\frac{\Tilde{E}-1}{\Tilde{E}}\right)^n L_n^0\left(\frac{\pi|z|^2}{\Tilde{E}-\Tilde{E}^2}\right)\gamma_{\Tilde{E}}(z)\\
    &=    \left(\frac{E-\frac{1}{2}}{E+\frac{1}{2}}\right)^n L_n^0\left(\frac{\pi |z|^2}{\frac{1}{4}-E^2}\right)\gamma_{E+\frac{1}{2}}(z),
\end{align*}
where we have used the scaling formula in the second equality.
\end{proof}
The eigenvalue formula, Equation \eqref{GaussFormula}, now follows directly once we note that \(g_M=\gamma_{M/2\pi}\). Note that the argument of the Laguerre function \(L_n^0\) is negative. This guarantees the positivity of the eigenvalues.
\section{Toeplitz operators on quantum Gabor spaces}\label{Sec:Toeplitz}

In this final section we will use the theory developed in this paper to study Toeplitz operators on the quantum Gabor space \(\mathfrak{V}_S(\mathcal{S}^2)\), as defined in Equation~\eqref{ToeplitzDef}. From the results in \cite{OpSTFT}, we know that these operators are unitarily equivalent to composition with a mixed-state localization operator, so we will therefore use quantum double orthogonality to study the spectral properties of these Toeplitz operators. In particular, the Toeplitz operator with mask \(F\) acting on \(\mathfrak{V}_S(\mathcal{S}^2)\) will be unitarily equivalent to composition with \(F\star (SS^*)\). Note that the operator \(SS^*\) is positive, and since \(S\) is Hilbert-Schmidt, \(SS^*\) will also be trace class. If we now also assume that \(F\in EFG(\R^{2d})\), the mixed-state localization operator satisfies the assumptions of Proposition \ref{OperatorOrth}. We translate this by unitarity to a result on Toeplitz operators.
\begin{prop}
    Let \(S\in\mathcal{S}^2\), \(F\in FG(\R^{2d})\), and \(\{\Psi_n\}_{n\in\N_0^d}\) be a collection of functions in \(\mathfrak{V}_S(\mathcal{S}^2)\). Then \(\{\Psi_n\}_{n\in\N_0^d}\) are eigenfunctions of the Toeplitz operator \(T_F\) on \(\mathfrak{V}_S(\mathcal{S}^2)\) if and only if they are complete and satisfy the identities \begin{align*}
        \int_{\R^{2d}}\mathrm{tr}\left( \Psi_n(z)\Psi_m^*(z)\right) \;dz=\delta_{n,m},
    \end{align*} and \begin{align*}
         \int_{\R^{2d}}F(z) \;\mathrm{tr}\left( \Psi_n(z)\Psi_m^*(z)\right)\;dz=c_n\delta_{n,m}.
    \end{align*}
    In this case, the \(n\)-th eigenvalue is \(c_n\).
\end{prop}
\begin{proof}
    Note first that since \(S\) is Hilbert-Schmidt, \(SS^*\) will be trace class, and since \(F\in FG(\R^{2d})\), \(F\star (SS^*)\) will be compact, and thus have an ONB of eigenvectors in \(L^2(\R^d)\), say \(\{\psi_n\}_{n\in\N_0^d}\). Thus the composition operator \(T\mapsto (F\star (SS^*))\circ T\) has an ONB of eigenoperators, namely \(\{\psi_n\otimes \psi_m\}_{m,n\in \N_0^d}\), as can be verified with a simple computation. Since \(T_F\) is unitarily equivalent to \(T\mapsto (F\star (SS^*))\circ T\) it will also have an ONB of eigenfunctions. So assume without loss of generality that the eigenfunctions of \(T_F\), say \(\{\Psi_n\}_{n\in\N_0^d}\), form an ONB of \(\mathfrak{V}_S(\mathcal{S}^2)\). We now proceed as before:\begin{align*}
        \lambda_n\delta_{n,m} &=\langle \lambda_n\Psi_n,\Psi_m\rangle=\langle T_F\Psi_n,\Psi_m\rangle=\int_{\R^{2d}} \mathrm{tr}\left(T_F\Psi_n(z)\Psi_m^*(z)\right)\;dz \\
        &=\int_{\R^{2d}} \mathrm{tr}\left(\int_{\R^{2d}} F(z')K(z,z')\Psi_n(z')\;dz'\Psi_m^*(z)\right)\;dz\\&=\int_{\R^{2d}} \int_{\R^{2d}} F(z')\;\mathrm{tr}\left(K(z,z')\Psi_n(z')\Psi_m^*(z)\right)\;dz'\;dz,
    \end{align*} 
    where we have used Fubini in the last line. Recalling the symmetry of the trace and using Fubini again, we moreover get \begin{align*}
        &=\int_{\R^{2d}} \int_{\R^{2d}} F(z')\;\mathrm{tr}\left(\Psi_m^*(z)K(z,z')\Psi_n(z')\right)\;dz'\;dz \\
        &=\int_{\R^{2d}} F(z')\;\mathrm{tr}\left(\int_{\R^{2d}} \Psi_m^*(z)K(z,z')\;dz\;\Psi_n(z')\right)\;dz'.
    \end{align*}
    Now we use the reproducing property of \(\mathfrak{V}_S(\mathcal{S}^2)\):\begin{align*}
        =\int_{\R^{2d}}  F(z')\;\mathrm{tr}\left(\Psi_m^*(z')\;\Psi_n(z')\right)\;dz'=\int_{\R^{2d}}  F(z')\;\mathrm{tr}\left(\Psi_n(z')\Psi_m^*(z')\right)\;dz'. 
    \end{align*}
    Thus we have verified that the functions \(\{\Psi_n\}_{n\in\N_0^d}\) satisfy the double orthogonality condition. \newline

    Now assume that we have a doubly orthogonal collection of functions \(\{\Psi_n\}_{n\in\N_0^d}\) in \(\mathfrak{V}_{S}(\mathcal{S}^2)\). Since they are an ONB, there are coefficients \(a_{m}\) such that \begin{align*}
        T_F(\Psi_n)=\sum_{m=0}^\infty a_{m} \Psi_m.
    \end{align*}
    The coefficients are \begin{align*}
        a_{m}=\int_{\R^{2d}}\;\mathrm{tr}\left(T_F(\Psi_n)(z)\Psi_m^*(z)\right)\;dz,
    \end{align*}
    which, as we have seen above, equals \begin{align*}
          \int_{\R^{2d}}  F(z)\;\mathrm{tr}\left(\Psi_n(z)\Psi_m^*(z)\right)\;dz.
    \end{align*}
    By double orthogonality, \(a_m=\lambda_n\delta_{n,m}\), and plugging this into the expansion, we see that \(\Psi_n\) is an eigenfunction of \(T_F\) for all \(n\in\N_0\).
\end{proof}
\begin{remark}
    Just like for mixed-state localization operators, the double orthogonality characterization extends to the case where \(F\in EFG(\R^{2d})\).
\end{remark}
A polarization argument now lets us carry Daubechies' theorem for polyradial operators over to a result on Toeplitz operators.
\begin{prop}
    Let \(F\in EFG(\R^{2d})\) be a polyradial mask, and let \(S\in\mathcal{S}^2\) be so that \(SS^*\) is a polyradial operator in \(\mathcal{S}^1\). Then the functions \(\{\mathfrak{V}_{S}(\phi_n\otimes \phi_m)\}_{n,m\in\N_0^d}\) are eigenfunctions of the Toeplitz operator with mask \(F\) acting on \(\mathfrak{V}_S(\mathcal{S}^2)\). The eigenvalue of \(\mathfrak{V}_{S}(\phi_n\otimes \phi_m)\) is the same as \(\phi_n\)'s eigenvalue under the mixed-state localization operator \(F\star(SS^*)\).
\end{prop}
\begin{proof}
    One could use double orthogonality to show this, but it is far easier to use the unitarily equivalent operator on \(\mathcal{S}^2\) and compute directly. We have \begin{align*}
        F\star (SS^*)\circ (\phi_n\otimes\phi_m)=(F\star (SS^*)\phi_n)\otimes\phi_m=(\lambda_n\phi_n)\otimes \phi_m=\lambda_n(\phi_n\otimes\phi_m),
    \end{align*}
    where the second inequality comes from Proposition \ref{OperatorDaub}.
\end{proof}
\begin{remark}
    As one might expect from the multiplicities of the eigenvalues, any function of the form \(\mathfrak{V}_S(\phi_n\otimes \psi)\) is an eigenfunction. However, we have stated the result with \(\psi=\phi_m\), since these functions also give an orthogonal basis of \(\mathfrak{V}_S(\mathcal{S}^2)\). Another way to state the result is that the \(n\)-th eigenspace of \(T_F\) is \(\mathrm{span}_{\C}\left(S^*\pi(z)^*\phi_n\right)\otimes L^2(\R^d)\).
\end{remark}
\section*{Acknowledgements}
The author thanks Franz Luef for many valuable discussions on the presented subject matter, and for useful comments on several drafts of this manuscript. The author also acknowledges the partial funding received from the project ``Pure Mathematics in Norway – Ren matematikk i Norge", funded by the Trond Mohn Foundation.

\printbibliography
\end{document}